\documentclass[10pt]{amsart}
\usepackage{amssymb}
\usepackage{epsfig}
\usepackage{url}
\usepackage[T1]{fontenc}
\usepackage{color}
\usepackage{tikz}
\usetikzlibrary{arrows}
\theoremstyle{plain}

\newtheorem{thm}{Theorem}[section]
\newtheorem{cor}[thm]{Corollary}
\newtheorem{lem}[thm]{Lemma}
\newtheorem{prop}[thm]{Proposition}
\newtheorem{rem}[thm]{Remark}
\newtheorem{ques}[thm]{Question}
\newtheorem{conj}[thm]{Conjecture}

\def\cal{\mathcal}
\def\bbb{\mathbb}
\def\op{\operatorname}
\renewcommand{\phi}{\varphi}

\newcommand{\N}{\bbb{N}}
\newcommand{\Z}{\bbb{Z}}
\newcommand{\Q}{\bbb{Q}}
\newcommand{\F}{\bbb{F}}

\newcommand{\eps}{\varepsilon}
\newcommand{\bs}{\backslash}
\newcommand{\sgn}{\mbox{sgn }}
\newcommand{\ord}{\emph{ord }}

\begin{document}

\title[Arithmetic properties]{Arithmetic properties of coefficients of power series expansion of $\prod_{n=0}^{\infty}\left(1-x^{2^{n}}\right)^{t}$ (with an Appendix by Andrzej Schinzel)}
\author{Maciej Gawron, Piotr Miska and Maciej Ulas}

\thanks{The research of the first and third author was supported by the grant of the Polish National Science Centre no. UMO-2012/07/E/ST1/00185}

\begin{abstract}
Let $F(x)=\prod_{n=0}^{\infty}(1-x^{2^{n}})$ be the generating function for the Prouhet-Thue-Morse sequence $((-1)^{s_{2}(n)})_{n\in\N}$. In this paper we initiate the study of the arithmetic properties of coefficients of the power series expansions of the function
$$
F_{t}(x)=F(x)^{t}=\sum_{n=0}^{\infty}f_{n}(t)x^{n}.
$$
For $t\in\N_{+}$ the sequence $(f_{n}(t))_{n\in\N}$ is the Cauchy convolution of $t$ copies of the Prouhet-Thue-Morse sequence. For $t\in\Z_{<0}$ the $n$-th term of the sequence  $(f_{n}(t))_{n\in\N}$ counts the number of representations of the number $n$ as a sum of powers of 2 where each summand can have one among $-t$ colors.

Among other things, we present a characterization of the solutions of the equations $f_{n}(2^k)=0$, where $k\in\N$, and $f_{n}(3)=0$. Next, we present the exact value of the 2-adic valuation of the number $f_{n}(1-2^{m})$ - a result which generalizes the well known expression concerning the 2-adic valuation of the values of the binary partition function introduced by Euler and studied by Churchhouse and others.

\end{abstract}

\maketitle

\section{Introduction}\label{Section1}

Let $n\in\N$ and by $s_2(n)$ denote the sum of (binary) digits function of $n$, i.e., if $n=\sum_{k=0}^m \eps_k 2^k$ with $\eps_k\in\{0,1\}$, is the unique expansion of $n$ in base $2$ then $s_2(n)=\sum_{k=0}^m \eps_k$. Next, let us define the Prouhet-Thue-Morse sequence (the PTM sequence for short) on the alphabet $\{-1,+1\}$ as ${\bf t}=(t_n)_{n\in\N}$, where $t_n=(-1)^{s_2(n)}$. The sequence ${\bf t}$ satisfies the following recurrence relations: $t_0=1$ and
\begin{equation*}
t_{2n}=t_n,\quad t_{2n+1}=-t_n
\end{equation*}
for $n\in\N$. The PTM sequence has many remarkable properties and found applications in combinatorics on words, analysis on manifolds, number theory and even physics \cite{AllSh}. One of the striking properties of the sequence ${\bf t}$ is the simple shape of the generating function $F(x)=\sum_{n=0}^{\infty} t_nx^n\in\Z[[x]]$. Indeed, from the recurrence relations we easily deduce the functional equation $F(x)=(1-x)F(x^2)$ and in consequence the identity
\begin{equation*}
F(x)=\prod_{n=0}^{\infty}\left(1-x^{2^n}\right).
\end{equation*}
The sequence ${\bf b}=(b_n)_{n\in\N}$ of coefficients of the related power series
\begin{equation*}
\frac{1}{F(x)}=\prod_{n=0}^{\infty}\frac{1}{1-x^{2^n}}=\sum_{n=0}^{\infty} b_nx^n
\end{equation*}
has also a strong combinatorial property. Indeed, the number $b_n$ counts the number of representations of a non-negative integer $n$ in the form
\begin{equation*}
n=\sum_{i=0}^k \eps_i 2^i,
\end{equation*}
where $k\in\N$ and $\eps_i\in\N$. One can easily prove that the sequence ${\bf b}$ satisfies: $b_0=b_1=1$ and
\begin{equation*}
b_{2n}=b_{2n-1}+b_n,\quad b_{2n+1}=b_{2n}
\end{equation*}
for $n\geq 1$. The above sequence is called the sequence of the binary partition function. It was introduced by Euler and was studied by Churchhouse \cite{Chu} (one can also consult the papers \cite{Mah, Bru, Knuth}).

From the discussion above we see that both ${\bf t}$ and ${\bf b}$ are sequences of coefficients of the power series expansion of $F_{t}(x)=F(x)^t$ for $t=1$ and $t=-1$, respectively. It is quite natural to ask: what can be proved about sequences of coefficients of $F_t(x)$ for other integer values of $t$? This question was our main motivation for writing this paper.

Let $t$ be a variable and consider the sequence ${\bf f}(t)=(f_n(t))_{n\in\N}$ of coefficients of the power series expansion of the function $F_t(x)=F(x)^t$, i.e.,
\begin{equation*}
F_t(x)=\prod_{n=0}^{\infty}\left(1-x^{2^n}\right)^t=\sum_{n=0}^{\infty} f_n(t)x^n.
\end{equation*}
From the definition of ${\bf f}(t)$ we see that for any given $t\in\Z$ the sequence ${\bf f}(t)$ is a sequence of integers. In the sequel we will study three closely related sequences. More precisely, in Section 2 we present some properties of the sequence ${\bf f}(t)$ treated as a sequence of polynomials with rational coefficients. This is only a prelude to our research devoted to the values of the polynomials $f_n$ at integer arguments.
Section 3 is devoted to the study of the sequence
\begin{equation*}
{\bf t}_m=(t_m(n))_{n\in\N},
\end{equation*}
where $m\in\N_+$ is fixed and $t_m(n)=f_n(m)$, i.e., $t_m(n)$ is just the value of the polynomial $f_n$ at $t=m$. We prove several properties of the sequence $t_m$ for certain values of $m$. In particular, in Theorem \ref{v2powerof2} we characterize the $2$-adic valuation of the sequence ${\bf t}_m$ for $m$ being a power of $2$ and $m=3$. In the second part of this section we concentrate on the study of arithmetic properties of the sequence ${\bf t}_m$ for $m=2$ and $m=3$. It is a simple observation that the sequence $t_2$ is closely related to the values of the Stern polynomials at $-2$. Moreover, we prove that the set of values of ${\bf t}_2$ is just $\Z\bs\{0\}$, which is the statement of Theorem \ref{valuesoft2} and that our sequence is log-concave, i.e., for each $n\in\N_+$ we have $t_2(n)^2>t_2(n-1)t_2(n+1)$ (Theorem \ref{concave}). We also characterize the set of those $n\in\N_+$ such that $t_3(n)=0$ (Theorem \ref{23zero}). This allows us to prove that there are infinitely many values of $n$ such that the polynomial $f_n(t)/t$ is reducible (Corollary \ref{reduc}). Section 4 is devoted to the study of the sequence ${\bf b}_m=(b_m(n))_{n\in\N}$, where $m\in\N_+$ is fixed and $b_m(n)=f_n(-m)$, i.e., $b_m(n)$ is just the value of the polynomial $f_n$ at $t=-m$. The sequence $b_m$ has a natural combinatorial interpretation. More precisely, the number $b_m(n)$ counts the number of representations
\begin{equation*}
n=\sum_{i=0}^k \eps_i 2^i,
\end{equation*}
where $\eps_i\in\N$ for $i\in\{0,...,k\}$ and each $\eps_i$ can have one among $m$ colors. We present several results concerning this family of sequences. In particular, we study the $2$-adic valuation of $b_m(n)$ and give a precise expression for $\nu_2(b_{2^k-1}(n))$, which allows us to deduce that the congruence $b_{2^k-1}(n)\equiv 0\pmod{16}$ is impossible (Theorem \ref{2valpowerof2-1}). We also study more closely the family of polynomials $(h_{i,k,m}(x))$, with $k\in\N$, $0\leq i<2^k$, $m\in\N_+$, which appear in the computation of the
expression for the generating function
\begin{equation*}
H _{i,k}(x)=H_{i,k,m}(x)=\sum_{n=0}^{\infty} b_m(2^kn+i)x^n=\frac{h_{i,k,m}(x)}{(1-x)^{km}}F_{-m}(x).
\end{equation*}
The obtained results allow us to prove several congruences of various types for certain sequences ${\bf b}_m$ (Theorem \ref{4.10}, Theorem \ref{4.13}). We also prove that for fixed $k\in\N$ and $0\leq i<2^k$ the sequence $(h_{i,k,m}(x))_{m\in\N}$ is a linear recurrence sequence of order $\leq 2^k$ (Theorem \ref{recurrence}).

In Section 5 we present some other results, questions and conjectures concerning sequences ${\bf t}_m$ and ${\bf b}_m$ for various values of $m\in\N_+$. We hope that the problems stated in this section will stimulate further research in the area.

Finally, in the Appendix, written by A. Schinzel, the proof of Conjecture \ref{sym} from Section 3 is presented together with other material concerning non-vanishing of $t_m(n)$.

\section{Arithmetic properties of the coefficients of $F_{t}(x)$}\label{Section2}

We start with the computation of a recurrence relation satisfied by the sequence ${\bf f}(t)=(f_{n}(t))_{n\in\N}$ and then introduce a related family of polynomials which will the main object of our study in this section. Let us put
\begin{equation*}
F_{t}(x)=F(x)^{t}=\prod_{n=0}^{\infty}\left(1-x^{2^{n}}\right)^{t}=\sum_{n=0}^{\infty}f_{n}(t)x^{n}.
\end{equation*}
During this paper we will treat all the power series formally, without considering their region of convergence. The function $F(t,x)$ satisfies the following functional equation $F_{t}(x)=(1-x)^{t}F_{t}(x^2)$. This functional equation allows us to find a pair of recurrence relations satisfied by the the sequence ${\bf f}(t)$.

We start with the following simple
\begin{lem}
We have the following identity
\begin{equation}\label{logPTM}
\op{log}F(x)=\sum_{n=1}^{\infty}\frac{1-2^{\nu_{2}(n)+1}}{n}x^{n},
\end{equation}
where $\nu_{2}(n)$ is the $2$-adic valuation of the integer $n$ and
\begin{equation*}
\op{log}(1+x)=\sum_{k=1}^{\infty}\frac{(-1)^{k-1}x^k}{k}.
\end{equation*}
\end{lem}
\begin{proof}
We use the expansion of the function $\op{log}(1-x)$ to obtain
\begin{align*}
\op{log}F(x)&=\sum_{n=0}^{\infty}\op{log}(1-x^{2^n})=-\sum_{n=0}^{\infty}\sum_{k=1}^{\infty}\frac{x^{2^nk}}{k}\\
            &=-\sum_{m=1}^{\infty}\sum_{l=0}^{\nu_2(m)}\frac{x^{2^l\cdot\frac{m}{2^l}}}{\frac{m}{2^l}}=-\sum_{m=1}^{\infty}\frac{(2^{\nu_2(m)+1}-1)}{m}x^m.
\end{align*}
\end{proof}

\begin{rem}
{\rm Using exactly the same type of reasoning, one can prove the following identity
$$
\log \prod_{n=0}^{\infty}\left(1-x^{k^{n}}\right)=\sum_{n=1}^{\infty}\frac{1-k^{\phi_{k}(n)+1}}{n}x^{n},
$$
where $\phi_{k}(n)$ is the highest power of $k$ which divides $n$.}
\end{rem}

As an application of the above lemma we get the following recurrence relation for the sequence ${\bf f}(t)$.
\begin{lem}\label{recforfprim}
Let $F_{t}(x)=\sum_{n=0}^{\infty}f_{n}(t)x^{n}$. Then $f_{0}(t)=1$ and for $n\geq 1$ we have
\begin{equation}\label{recfprim}
f_{n}(t)=\frac{t}{n}\sum_{k=0}^{n-1}(1-2^{\nu_{2}(n-k)+1})f_{k}(t).
\end{equation}
\end{lem}
\begin{proof}
We have the identity $\log F_{t}(x)=t\log F(x)$. Taking derivative of both sides with respect to $x$ and using the expansion (\ref{logPTM}), we get
$$
\frac{F'_{t}(x)}{F_{t}(x)}=t\sum_{n=1}^{\infty}(1-2^{\nu_{2}(n)+1})x^{n-1}.
$$
Multiplying both sides by $F_{t}(x)=\sum_{n=0}^{\infty}f_{n}(t)x^{n}$ and replacing $n$ by $n+1$ in the sum on the right side and in $F'_{t}(x)$ we get
\begin{align*}
\sum_{n=0}^{\infty}(n+1)f_{n+1}(t)x^{n}&=t\left(\sum_{n=0}^{\infty}(1-2^{\nu_{2}(n+1)+1})x^{n}\right)\left(\sum_{n=0}^{\infty}f_{n}(t)x^{n}\right)\\
                                     &=t\sum_{n=0}^{\infty}\left(\sum_{k=0}^{n}(1-2^{\nu_{2}(n-k+1)+1})f_{k}(t)\right)x^{n}.
\end{align*}
Comparing now the coefficients on both sides of the above equality and replacing $n$ by $n-1$, we get the identity from the statement of our lemma.
\end{proof}
Using other functional equations satisfied by $F_{t}(x)$, we can deduce other recurrence relations.

\begin{lem}\label{recforf}
The sequence ${\bf f}(t)$ satisfies the following recurrence relations:
\begin{enumerate}
 \item $f_{0}(t)=1$ and
\begin{equation*}
f_{n}(t)=-\sum_{k=0}^{n-1}{t+n-k-1\choose n-k}f_{k}(t)+\chi_{2}(n)f_{\frac{n}{2}}(t),
\end{equation*}
where $\chi_{2}(n)=(1+(-1)^{n})/2$;
\item $f_{0}(t)=1$ and
\begin{equation*}
f_{n}(t)=\sum_{k=0}^{\left\lfloor\frac{n}{2}\right\rfloor}{n-2k-1-t\choose n-2k}f_{k}(t),
\end{equation*}
\end{enumerate}
\end{lem}

\begin{proof}
In order to prove the first recurrence relation for the sequence ${\bf f}(t)$ we rewrite the functional equation for the function $F_{t}(x)$ in the following form:
\begin{align*}
F_{t}(x^2)=\frac{1}{(1-x)^{t}}F_{t}(x)&=\left(\sum_{n=0}^{\infty}{t+n-1\choose n}x^n\right)\left(\sum_{n=0}^{\infty}f_{n}(t)x^n\right)\\
&=\sum_{n=0}^{\infty}\left(\sum_{k=0}^{n}{t+n-k-1\choose n-k}f_{k}(t)\right)x^{n}.
\end{align*}
However,
$$
F_{t}(x^2)=\sum_{n=0}^{\infty}f_{n}(t)x^{2n}=\sum_{n=0}^{\infty}\chi_{2}(n)f_{\frac{n}{2}}(t)x^{n},
$$
and thus comparing the coefficients near $x^{n}$ in the identity $F_{t}(x^2)=(1-x)^{-t}F_{t}(x)$ and performing simple manipulations we get the first recurrence relation for the sequence ${\bf f}(t)$.

In order to prove the second recurrence relation we compute
\begin{align*}
F_{t}(x)&=\left(\frac{1}{1-x}\right)^{-t}\sum_{n=0}^{\infty}f_{n}(t)x^{2n}=\left(\sum_{n=0}^{\infty}{n-1-t\choose n}x^n\right)\left(\sum_{n=0}^{\infty}\chi_{2}(n)f_{\frac{n}{2}}(t)x^{n}\right)\\
&=\sum_{n=0}^{\infty}\left(\sum_{k=0}^{n}{n-k-1-t\choose n-k}\chi_{2}(k)f_{\frac{k}{2}}(t)\right)x^{n}\\
&=\sum_{n=0}^{\infty}\left(\sum_{k=0}^{\left\lfloor\frac{n}{2}\right\rfloor}{n-2k-1-t\choose n-2k}f_{k}(t)\right)x^{n}.
\end{align*}
Comparing now the coefficients on both sides of the above identity, we get the second recurrence relation from the statement of our lemma.
\end{proof}

Lemma \ref{recforfprim} and Lemma \ref{recforf} show us that if we fix $n\in\N$, then the expression $f_{n}(t)$ is a polynomial with respect to $t$. The first terms of the sequence $(f_{n}(t))_{n\in\N}$ are:
\begin{align*}
f_{0}(t)&=1,\\
f_{1}(t)&=-t,\\
f_{2}(t)&=\frac{1}{2} (t-3) t,\\
f_{3}(t)&=-\frac{1}{6} t \left(t^2-9 t+2\right),\\
f_{4}(t)&=\frac{1}{24} t \left(t^3-18 t^2+35 t-42\right),\\
f_{5}(t)&=-\frac{1}{120} t \left(t^4-30 t^3+155 t^2-270 t+24\right).
\end{align*}

As a consequence of the recurrence relation for ${\bf f}(t)$, we get the following properties of the sequence ${\bf f}(t)$.

\begin{lem}
 We have:
\begin{enumerate}
\item $\op{deg}f_{n}(t)=n$;
\item $f_{0}(0)=1$ and $f_{n}(0)=0$ for $n>0$;
\item Let us write
$$
f_{n}(t)=\sum_{i=0}^{n}a(i,n)t^{i}.
$$
Then $a(0,0)=1, a(0,n)=0$ for $n\in\N_{+}$ and for $i\in\{0,\ldots,n-1\}$ we have
\begin{equation}\label{recfora}
a(i+1,n)=\frac{1}{n}\sum_{j=i}^{n-1}(1-2^{\nu_{2}(n-j)+1})a(i,j).
\end{equation}
In particular we have the following equalities:
\begin{align*}
a(n,n)&=\frac{(-1)^{n}}{n!},\\
a(n-1,n)&=\frac{(-1)^{n+1}}{2!(n-2)!}3,\;n\geq 2,\\
a(n-2,n)&=\frac{(-1)^{n}}{4!(n-3)!}(27n-73),\;n \geq 3,\\
a(1,n)&=\frac{1-2^{\nu_{2}(n)+1}}{n},\;n\geq 1.\\
\end{align*}

\item The sequence ${\bf f}(t)$ satisfies the following addition formula:
$$
f_{n}(t_{1}+t_{2})=\sum_{k=0}^{n}f_{k}(t_{1})f_{n-k}(t_{2}),
$$
where $t_{1}, t_{2}$ are variables.
\end{enumerate}

\end{lem}
\begin{proof}
The first and the second statement follow immediately from Lemma \ref{recforfprim}.

In order to prove the third statement we use Lemma \ref{recforfprim} one more time. For $n\geq 1$ we have the following equalities:
\begin{align*}
f_{n}(t)=\sum_{i=0}^{n}a(i,n)t^{i}&=\frac{t}{n}\sum_{j=0}^{n-1}\left(1-2^{\nu_{2}(n-j)+1}\right)\sum_{i=0}^{k}a(i,j)t^{i}\\
&=\frac{1}{n}\sum_{i=0}^{n-1}\left(\sum_{j=i}^{n-1}(1-2^{\nu_{2}(n-j)+1})a(i,j)\right)t^{i+1}\\
&=\frac{1}{n}\sum_{i=1}^{n}\left(\sum_{j=i-1}^{n-1}(1-2^{\nu_{2}(n-j)+1})a(i-1,j)\right)t^{i}.
\end{align*}
By comparing the coefficients of the polynomial $f_{n}(t)$ and the polynomial given by the last expression, we get the result (after the change of variables $i\rightarrow i+1$).

In order to prove the expression for $a(n,n)$ we use Lemma \ref{recforfprim} one more time. We immediately deduce the equality
$$
a(n,n)=\frac{1}{n}(1-2^{\nu_{2}(n-(n-1))+1})a(n-1,n-1)=-\frac{1}{n}a(n-1,n-1).
$$
Using simple induction and the identity $a(0,0)=1$, we get the expression for $a(n,n)$.

Next, we have $a(1,2)=-3/2$ and for $n\geq 2$ by (\ref{recfora}) with $i=n-2$ we get
$$
a(n-1,n)=-\frac{3}{n}a(n-2,n-2)-\frac{1}{n}a(n-2,n-1)=-\frac{3}{n}\frac{(-1)^{n}}{(n-2)!}-\frac{1}{n}a(n-2,n-1).
$$
Using now simple induction, we easily get the expression for $a(n-1,n)$ presented in the statement of our lemma.

Because exactly the same technique as used for the proof of expressions for $a(n,n)$ and $a(n-1,n)$ can be applied in order to compute $a(n-2,n)$, we omit the proof and leave the simple details for the reader.

Finally, in order to get the addition formula we notice that it is a simple consequence of the formal identity $F_{t_{1}}(x)F_{t_{2}}(x)=F_{t_{1}+t_{2}}(x)$.
\end{proof}

\begin{rem}
{\rm Although we were unable to find the general formula for the coefficients $a(n-k,n)$ for all $k\geq 4$ and $n\geq k+1$, it is an easy (but tedious) exercise to prove that for fixed $k$ we have
$$
a(n-k,n)=\frac{(-1)^{n+k}}{(2k)!(n-k-1)!}W_{k}(n),\;n\geq k+1,
$$
where $W_{k}\in\Z[n]$ is of degree $k-1$.

Using this observation one can compute polynomials $W_{k}(n)$ for several values of $k$:
\begin{align*}
W_{3}(n)&=45(9 n^2-73 n+176),\\
W_{4}(n)&=7(1215 n^3-19710 n^2+121685 n-266398),\\
W_{5}(n)&=945 \left(243 n^4-6570 n^3+74165 n^2-394878 n+805440\right),\\
W_{6}(n)&=165 \left(45927 n^5-1862595 n^4+33070275 n^3-310359581 n^2+1497391014 n-2916611728\right).
\end{align*}

}
\end{rem}

We introduce the family of polynomials ${\bf g}(t)=(g_{n}(t))_{n\in\N}$, where
$$
g_{n}(t)=n!f_{n}(t).
$$
As a consequence of the recurrence relation for ${\bf f}(t)$, we get the recurrence relation satisfied by the sequence ${\bf g}(t)$ in the following form:
$$
g_{0}(1)=1,\quad g_{n}(t)=t\sum_{k=0}^{n-1}(1-2^{\nu_{2}(n-k)+1})\frac{(n-1)!}{k!}g_{k}(t).
$$
In particular $g_n(t)\in\Z[t]$ for each $n\in\N$.


We have the following result concerning the factorization of $g_{n}(t)$ modulo $p$, where $p$ is a prime number.
\begin{thm}\label{congruence}
Let $n\in\N$. Then
\begin{equation*}
g_{n}(t)\equiv g_{n \bmod p}(t)(t-t^{p})^{\left\lfloor\frac{n}{p}\right\rfloor}\pmod{p}.
\end{equation*}
\end{thm}
\begin{proof}
Let $p$ be a prime number. We proceed by induction on $n$. Our factorization is clearly true for $n\leq p-1$. If $n=p$ then $p\mid g_p(j)$ for any $j\in\Z$. Since $g_p(t)\in\Z[t]$, $\deg g_p(t)=p$ and the leading coefficient of $g_p(t)$ is $(-1)^p$, we thus have
\begin{equation*}
g_p(t)\equiv (-1)^p\prod_{a=0}^{p-1} (t-a)\equiv t-t^p\pmod{p}.
\end{equation*}
Let us consider the case $n=pm+i$ for some $i\in\{1,\ldots,p\}$. Observe that
\begin{equation*}
\frac{(n-1)!}{k!}=\frac{(pm+i-1)!}{k!}\equiv
\begin{cases}
\begin{array}{lll}
0, &                  & \mbox{if}\;k\leq pm-1 \\
\frac{(i-1)!}{j!}, &  & \mbox{if}\;k=pm+j,\;j\in\{0,1,\ldots,i-1\}
\end{array}
\end{cases}\pmod{p}.
\end{equation*}
We have the following chain of congruences $\bmod p$ for $i\in\{1,\ldots,p-1\}$:
\begin{align*}
g_{n}(t)&\equiv g_{pm+i}(t)\equiv t\sum_{k=0}^{pm+i-1}(1-2^{\nu_{2}(pm+i-k)+1})\frac{(pm+i-1)!}{k!}g_{k}(t)\\
        &\equiv t\sum_{k=0}^{pm-1}(1-2^{\nu_{2}(pm+i-k)+1})\frac{(pm+i-1)!}{k!}g_{k}(t)+t\sum_{k=pm}^{pm+i-1}(1-2^{\nu_{2}(pm+i-k)+1})\frac{(pm+i-1)!}{k!}g_{k}(t)\\
        &\equiv t\sum_{k=pm}^{pm+i-1}(1-2^{\nu_{2}(pm+i-k)+1})\frac{(pm+i-1)!}{k!}g_{k}(t)\\
        &\equiv t\sum_{j=0}^{i-1}(1-2^{\nu_{2}(i-j)+1})\frac{(i-1)!}{j!}g_{j}(t)(t-t^{p})^{m}\\
        &\equiv (t-t^{p})^{m}t\sum_{j=0}^{i-1}(1-2^{\nu_{2}(i-j)+1})\frac{(i-1)!}{j!}g_{j}(t)\\
        &\equiv g_{i}(t)(t-t^{p})^{m}\equiv g_{n\bmod{p}}(t)(t-t^{p})^{\lfloor\frac{n}{p}\rfloor}\pmod{p}.
\end{align*}
If $i=p$ then in the same way we obtain
$$g_{n}(t)\equiv g_p(t)(t-t^{p})^{\lfloor\frac{n}{p}\rfloor}\equiv (t-t^{p})^{\lfloor\frac{n}{p}\rfloor+1}\pmod{p}.$$
Our result follows.
\end{proof}

\section{Arithmetic properties of the sequence $(f_{n}(t))_{n\in\N}$ with $t\in\N_{+}$}\label{Section3}

In this section we consider the sequence $(f_{n}(t))_{n\in\N}$ with a fixed positive integer $t$. We thus write $t=m$ for $m\in\N_{+}$ and define
$$
t_{m}(n):=f_{n}(m).
$$
Moreover, we put ${\bf t}_{m}=(t_{m}(n))_{n\in\N}$. In particular ${\bf t}_{1}=(t_{1}(n))_{n\in\N}=((-1)^{s_{2}(n)})_{n\in\N}=(t_{n})_{n\in\N}$ is the Prouhet-Thue-Morse sequence. It is clear that ${\bf t}_{m}$ is the sequence obtained from the convolution of $m$ copies of the Prouhet-Thue-Morse sequence, i.e.,
\begin{equation}\label{deftm}
t_{m}(n)=\sum_{i_{1}+i_{2}+\ldots+i_{m}=n}(-1)^{\sum_{k=1}^{m}s_{2}(i_{k})}.
\end{equation}

\subsection{Results concerning the computation of the $2$-adic valuation of $t_{m}(n)$}

This subsection is devoted to the presentation of the results concerning the explicit computation of the 2-adic valuation of the sequence ${\bf t}_{m}$ for $m=2^{k}$ and $m=3$.

From the functional equation for $F_{m}(x)$ we easily deduce the following useful result.

\begin{lem}\label{rectmn}
Let $m$ be a positive integer. Then $t_m(0)=1$, $t_m(1)=-m$ and
\begin{equation*}
t_m(2n)=\sum_{j=0}^{\left\lfloor\frac{m}{2}\right\rfloor} {m\choose 2j}t_m(n-j),\quad t_m(2n+1)=-\sum_{j=0}^{\left\lfloor\frac{m-1}{2}\right\rfloor} {m\choose 2j+1}t_m(n-j),
\end{equation*}
where we set put $t_m(n)=0$ for $n<0$.
\end{lem}
\begin{proof}
Let us expand $F_m(x)$ using the functional equation $F_m(x)=(1-x)^mF_m(x^2)$.
\begin{equation*}
\begin{split}
F_m(x) &=(1-x)^mF_m(x^2)=\left(\sum_{j=0}^m {m\choose j}(-1)^jx^j\right)\left(\sum_{k=0}^{\infty} t_m(k)x^{2k}\right) \\
&=\left(\sum_{j=0}^{\left\lfloor\frac{m}{2}\right\rfloor} {m\choose 2j}x^{2j}-\sum_{j=0}^{\left\lfloor\frac{m-1}{2}\right\rfloor} {m\choose 2j+1}x^{2j+1}\right)\left(\sum_{k=0}^{\infty} t_m(k)x^{2k}\right) \\
&=\sum_{n=0}^{\infty}\left[\left(\sum_{j=0}^{\left\lfloor\frac{m}{2}\right\rfloor} {m\choose 2j}t_m(n-j)\right)x^{2n}-\left(\sum_{j=0}^{\left\lfloor\frac{m-1}{2}\right\rfloor} {m\choose 2j+1}t_m(n-j)\right)x^{2n+1}\right]
\end{split}
\end{equation*}
Comparing coefficients of the first and last expression, we obtain the recurrence relations in the statement of our lemma.
\end{proof}

\begin{lem}\label{parityrest}
 Let $m$ be a positive integer. Then $F_{m}(x)\equiv (1+x)^{-m}\pmod{2}$. In particular,
\begin{equation}\label{parityt}
 t_{m}(n)\equiv {n+m-1\choose m-1}\pmod{2}
\end{equation}
for each $n\in\N$.
\end{lem}
\begin{proof}
In order to prove our result let us recall the identity $\prod_{n=0}^{\infty}(1+x^{2^{n}})=\frac{1}{1-x}$. We thus have
\begin{align*}
F_{m}(x)&\equiv \prod_{n=0}^{\infty}\left(1+x^{2^{n}}\right)^{m}\equiv \left(\prod_{n=0}^{\infty}\left(1+x^{2^{n}}\right)\right)^{m}\\
        &\equiv \frac{1}{(1-x)^{m}}\equiv \sum_{n=0}^{\infty}{n+m-1\choose m-1}x^{n}\pmod{2}.
\end{align*}
This proves the first part of our lemma. In order to get the second part we compare the coefficients (modulo 2) of $x^{n}$ on both sides of the first and last term in the above congruence. Our result follows.
\end{proof}

We can strengthen the result above for $m=2^k$, $k\in\N$. Namely

\begin{thm}\label{v2powerof2}
Let $k\in\N$. Then $\nu_2(t_{2^k}(n))=\nu_2\left(n+2^k-1\choose 2^k-1\right)$ for each $n\in\N$. In other words,
\begin{equation*}
\nu_2(t_{2^k}(2^kn+j))=
\begin{cases}
k-\nu_2(j)+\nu_2(n+1) & \mbox{ when } j\in\{1,...,2^k-1\} \\
0 & \mbox{ when } j=0
\end{cases}
\end{equation*}
for $n\in\N$.
\end{thm}

\begin{proof}
Let us note that
\begin{align*}
 \nu_2\left((2^kn+j)+2^k-1\choose 2^k-1\right)&=\nu_2\left(2^k(n+1)+j-1\choose 2^k-1\right)\\
                                              &=\nu_2\left(\frac{2^k}{j}{2^k(n+1)+j\choose 2^k}\right) \\
                                              &=k-\nu_2(j)+s_2(2^k)+s_2(2^kn+j)-s_2(2^k(n+1)+j) \\
                                              &=k-\nu_2(j)+1+s_2(n)+s_2(j)-s_2(n+1)-s_2(j)\\
                                              &=k-\nu_2(j)+\nu_2\left(n+1\choose n\right)=k-\nu_2(j)+\nu_2(n+1),
\end{align*}
where we used Legendre's formula $\nu_2\left(n!\right)=n-s_2(n)$ (see \cite{Le}). By the above equality and the fact that each nonnegative integer can be represented in the form $2^kn+j$ for some $n\in\N$ and $j\in\{0,1,...,2^k-1\}$, it suffices to show by induction on $2^kn+j\in\N$ that $\nu_2(t_{2^k}(2^kn+j))=k-\nu_2(j)+\nu_2(n+1)$ for $2^kn+j\geq -2^k$ (we recall that $t_m(n)=0$ for $n<0$). Clearly the statement is true for $2^kn+j\leq 1$. Let us compute the $2$-adic valuation of the numbers $t_{2^k}(2^{k+1}n+2j)$ and $t_{2^k}(2^{k+1}n+2j+1)$, where $n\in\N$ and $j\in\{0,1,...,2^k-1\}$ and $2^{k+1}n+2j\geq 2$. By Lemma \ref{rectmn}, we have
\begin{equation}\label{evenrec}
t_{2^k}(2^{k+1}n+2j)=\sum_{i=0}^{2^{k-1}} {2^k\choose 2i}t_{2^k}(2^kn+j-i)
\end{equation}
If $j\in\{0,2^{k-1}\}$ then only the summand for $i=0$ is odd, thus $t_{2^k}(2^{k+1}n+2j)$ is odd. Let $j\not\in\{0,2^{k-1}\}$. Then we compute $2$-adic valuation of each summand of the sum   in (\ref{evenrec}). We start with the $2$-adic valuation of ${2^k\choose 2i}$.
\begin{align*}
\nu_2\left({2^k\choose 2i}\right)&=\nu_2\left(\frac{2^k}{2i}{2^k-1\choose 2i-1}\right)\\
                                 &=k-1-\nu_2(i)+s_2(2i-1)+s_2(2^k-2i)-s_2(2^k-1)=k-1-\nu_2(i), i>0
\end{align*}
By the above equality and the induction hypothesis, we obtain
\begin{equation*}
\nu_2\left({2^k\choose 2i}t_{2^k}(2^kn+j-i)\right)=
\begin{cases}
k-\nu_2(j)+\nu_2(n+1), & \mbox{ if } i=0 \\
2k-1-\nu_2(i)-\nu_2(j-i)+\nu_2(n+1), & \mbox{ if } 0<i<j \\
k-1-\nu_2(j), & \mbox{ if } i=j \\
2k-1-\nu_2(i)-\nu_2(j-i)+\nu_2(n), & \mbox{ if } i>j
\end{cases}.
\end{equation*}
Let us notice that
\begin{equation}\label{ineq}
k-\nu_2(j)+\nu_2(n+1)\leq 2k-1-\nu_2(i)-\nu_2(j-i)+\nu_2(n+1),
\end{equation}
when $0<i\leq 2^{k-1}$ and $i\neq j$. Indeed, the inequality (\ref{ineq}) is equivalent to
\begin{equation}\label{ineq2}
\nu_2(i)+\nu_2(j-i)\leq k-1+\nu_2(j).
\end{equation}
If $\nu_2(i)\leq\nu_2(j)$ then (\ref{ineq2}) holds, since $0<|j-i|\leq 2^{k-1}$ and in consequence $\nu_2(j-1)\leq k-1$. If $\nu_2(i)>\nu_2(j)$ then $\nu_2(j-i)=\nu_2(j)$ and $\nu_2(i)\leq k-1$, since $0<i\leq 2^{k-1}$. Now we see that the $j$th summand of the sum in (\ref{evenrec}) has the $2$-adic valuation less than any other summand. Indeed,
\begin{equation*}
k-1-\nu_2(j)<k-\nu_2(j)+\nu_2(n+1)\leq 2k-1-\nu_2(i)-\nu_2(j-i)+\nu_2(n+1)
\end{equation*}
and
\begin{equation*}
k-1-\nu_2(j)<k-\nu_2(j)+\nu_2(n)\leq 2k-1-\nu_2(i)-\nu_2(j-i)+\nu_2(n).
\end{equation*}
We thus infer that $\nu_2(t_{2^k}(2^{k+1}n+2j))=k-\nu_2(2j)=k-\nu_2(2j)+\nu_2(2n+1)$ when $0<j<2^{k-1}$. If $j>2^{k-1}$ then by (\ref{ineq}) we know that the $0$th summand of the sum in (\ref{evenrec}) has the least $2$-adic valuation. It suffices to check for which $i\in\{1,...,2^{k-1}\}$ the $i$th summand has the same $2$-adic valuation as the $0$th one, or, in other words, we have equality in (\ref{ineq2}). Equality in (\ref{ineq2}) holds only if $i=2^{k-1}$ or $i=j-2^{k-1}$. Hence in the sum in (\ref{evenrec}) there are three summands with minimal $2$-adic valuation. As a consequence, $\nu_2(t_{2^k}(2^{k+1}n+2j))=k-\nu_2(j)+\nu_2(n+1)=k-\nu_2(2j)+\nu_2(2n+2)$. We are left with the computation of $t_{2^k}(2^{k+1}n+2j+1)$. By Lemma \ref{rectmn}, we have
\begin{equation}\label{oddrec}
t_{2^k}(2^{k+1}n+2j+1)=\sum_{i=0}^{2^{k-1}-1} {2^k\choose 2i+1}t_{2^k}(2^kn+j-i)
\end{equation}
We start with the $2$-adic valuation of ${2^k\choose 2i+1}$ when $0\leq i<2^{k-1}$.
\begin{align*}
\nu_2\left({2^k\choose 2i+1}\right)&=\nu_2\left(\frac{2^k}{2i+1}{2^k-1\choose 2i}\right)\\
                                   &=k-\nu_2(2i+1)+s_2(2i)+s_2(2^k-2i-1)-s_2(2^k-1)=k.
\end{align*}
By the above equality and the induction hypothesis we obtain
\begin{equation*}
\nu_2\left({2^k\choose 2i+1}t_{2^k}(2^kn+j-i)\right)=
\begin{cases}
2k-\nu_2(j-i)+\nu_2(n+1), & \mbox{ if } 0\leq i<j \\
k, & \mbox{ if } i=j \\
2k-\nu_2(j-i)+\nu_2(n), & \mbox{ if } i>j
\end{cases}.
\end{equation*}
Since $0\leq i,j<2^k$, thus for $i\neq j$ we have $0<|j-i|<2^k$. This implies $\nu_2(j-i)<k$ and hence $k<2k-\nu_2(j-i)$. This means that for $j<2^{k-1}$ the $j$th summand in (\ref{oddrec}) has the $2$-adic valuation less than any other summand and $\nu_2(t_{2^k}(2^{k+1}n+2j+1))=k=k-\nu_2(2j+1)+\nu_2(2n+1)$. If $j\geq 2^{k-1}$ then the $j-2^{k-1}$th summand in (\ref{oddrec}) has the $2$-adic valuation less than any other summand and $\nu_2(t_{2^k}(2^{k+1}n+2j+1))=k+1+\nu_2(n+1)=k-\nu_2(2j+1)+\nu_2(2n+2)$. This finishes the proof.
\end{proof}

\begin{cor}
For each $k\in\N$ the sequence $(\nu_2(t_{2^k}(n)))_{n\in\N}$ is $2$-regular, i.e., the $\Z$-submodule of $\Z^{\N}$ generated by sequences $(\nu_2(t_{2^k}(2^ln+j)))_{n\in\N}$, $l\in\N$, $j\in\{0,1,...,2^l-1\}$, is finitely generated (see \cite{AllSh1}).
\end{cor}

\begin{proof}
It suffices to show that the sequences $(1)_{n\in\N}$ and $(\nu_2(n+1))_{n\in\N}$ lie in the $\Z$-submodule of $\Z^{\N}$ generated by sequences $(\nu_2(t_{2^k}(2^ln+j)))_{n\in\N}$, $l\geq k$, $j\in\{0,1,...,2^l-1\}$, and generate these sequences.
Obviously, $(1)_{n\in\N}=(\nu_2(t_{2^k}(2^kn+1)))_{n\in\N}-(\nu_2(t_{2^k}(2^kn+2)))_{n\in\N}$ and $(\nu_2(n+1))_{n\in\N}$ can be written as $(\nu_2(t_{2^k}(2^kn+2^{k-1})))_{n\in\N}-(1)_{n\in\N}$.
Now, we prove by induction on $l\geq k$ that $(\nu_2(t_{2^k}(2^ln+j)))_{n\in\N}$ is of the form $(\alpha+\beta\nu_2(n+1))_{n\in\N}$, where $\alpha\in\N$ and $\beta\in\{0,1\}$. This statement is true for $l=k$ by Theorem \ref{v2powerof2}. For $l>k$ we write $j=2^{l-1}s+j'$, where $s\in\{0,1\}$ and $0\leq j'\leq 2^{l-1}-1$. Then by induction hypothesis we get the following
\begin{align*}
& (\nu_2(t_{2^k}(2^ln+j)))_{n\in\N}=(\nu_2(t_{2^k}(2^{l-1}(2n+s)+j')))_{n\in\N}\\
& =(\alpha+\beta\nu_2(2n+s+1))_{n\in\N}=(\alpha+\beta s+\beta s\nu_2(n+1))_{n\in\N},
\end{align*}
where the last equality holds because $\nu_2(2n+s+1)=s+s\nu_2(n+1)$ for $n\in\Z$ and $s\in\{0,1\}$.
\end{proof}

We can also describe the $2$-adic valuation of the numbers $t_3(n)$, $n\in\N$. We start with the following simple lemma.

\begin{lem}\label{8times}
For each $n\in\N$ we have
\begin{equation*}
t_3(4n+2)=8t_3(n-1),\quad t_3(4n+3)=8t_3(n),
\end{equation*}
where $t_3(-1)=0$.
\end{lem}

\begin{proof}
It suffices to use the recurrence for the sequence $(t_3(n))_{n\in\N}$ twice.
\begin{align*}
t_3(4n+3)&=-3t_3(2n+1)-t_3(2n)=9t_3(n)+3t_3(n-1)-t_3(n)-3t_3(n-1)=8t_3(n),\\
t_3(4n+2)&=t_3(2n+1)+3t_3(2n)=-3t_3(n)-t_3(n-1)+3t_3(n)+9t_3(n-1)=8t_3(n-1).
\end{align*}
\end{proof}

\begin{prop}\label{v2for3}
For each $n\in\N$ the following equalities hold:
\begin{align*}
\nu_2(t_3(4n))=\nu_2(t_3(4n+1))&=0,&\\
\nu_2(t_3(4n+3))=\nu_2(t_3(4n+6))&=3+\nu_2(t_3(n)),
\end{align*}
where in the second equality we assume that $t_3(n)=0$ for $n<0$.
\end{prop}

\begin{proof}
The numbers $t_3(4n)$ and $t_3(4n+1)$ are odd by Lemma \ref{parityrest}. For the proof of the equality $\nu_2(t_3(4n+3))=\nu_2(t_3(4n+6))=3+\nu_2(t_3(n))$ we use Lemma \ref{8times}.
\end{proof}

One can prove by induction on $n\in\N_+$ that every positive integer $n$ can be uniquely written in the form $n=\sum_{j=0}^d 4^ja_j$, where $a_j\in\{0,1,3,6\}$ for $j<d$ and $a_d\in\{1,2,3,6\}$. Then the $2$-adic valuation of $t_3(n)$, $n\in\N_+$ can be described in the following way.

\begin{thm}
For each $n\in\N_+$ there holds
\begin{equation*}
\nu_2(t_3(n))=
\begin{cases}
+\infty, & \mbox{ if } a_d=2 \mbox{ and } a_j\in\{3,6\} \mbox{ for } j<d\\
3k, & \mbox{ if } k=\max\{l\in\{1,...,d+1\}: a_j\in\{3,6\}\mbox{ for } j<l\} \mbox{ and } (a_d\neq 2\mbox{ or } k<d)
\end{cases}.
\end{equation*}
\end{thm}

\begin{proof}
The proof will be performed by induction on $d$. If $d=0$ then $n\leq 6$ and we check the statement of our theorem one by one. If $d>0$ then we write $n=4n'+a_0$. Then $n'=\sum_{j=0}^{d-1} 4^jb_j$, where $b_j=a_{j+1}$ for $j\in\{0,...,d-1\}$. If $a_0\in\{0,1\}$ then by Proposition \ref{v2for3} the number $t_3(n)$ is odd and our assertion follows. If $a_0\in\{3,6\}$ then we use Proposition \ref{v2for3} and the induction hypothesis to obtain the following:
\begin{align*}
\nu_2(t_3(n))&=3+\nu_2(t_3(n'))\\
&=\begin{cases}
+\infty, & \mbox{ if } b_{d-1}=2 \mbox{ and } b_j\in\{3,6\} \mbox{ for } j\in\{1,...,d-1\}\\
3+3k, & \mbox{ if } b_{d-1}\neq 2 \mbox{ and } k=\max\{l\in\{0,...,d\}: b_j\in\{3,6\}\mbox{ for } j<l\}
\end{cases}\\
&= \begin{cases}
+\infty, & \mbox{ if } a_d=2 \mbox{ and } a_j\in\{3,6\} \mbox{ for } j<d\\
3k, & \mbox{ if } a_d\neq 2 \mbox{ and } k=\max\{l\in\{1,...,d+1\}: a_j\in\{3,6\}\mbox{ for } j<l\}
\end{cases}.
\end{align*}
\end{proof}

\subsection{Unboundedness of ${\bf t}_{m}$ for $m=2^{k}$ and $m=3$}

As an application of Lemma \ref{rectmn} we get:

\begin{thm}
If $m\in\N_{\geq 2}$ then we have
\begin{equation*}
t_{m}(n)=O(n^{\frac{m}{2}})
\end{equation*}
for each $n\in\N$.
\end{thm}
\begin{proof}
We will prove by induction that $$|t_m(n)|\leq mn^{\frac{m}{2}}.$$ Clearly, the above inequality holds for $n\in\{0,1\}$. If $n>1$ and is even then we write $n=2n'$ for some $n'\in\N$. We use Lemma \ref{rectmn} and the induction hypothesis (we recall that $t_m(n)=0$ for $n<0$).
\begin{align*}
|t_m(n)| & =|t_m(2n')|=\left|\sum_{j=0}^{\lfloor\frac{m}{2}\rfloor} {m\choose 2j}t_m(n'-j)\right|\leq\sum_{j=0}^{\lfloor\frac{m}{2}\rfloor} {m\choose 2j}|t_m(n'-j)|\\
& <\sum_{j=0}^{\lfloor\frac{m}{2}\rfloor} {m\choose 2j}m(n')^{\frac{m}{2}}=m(2n')^{\frac{m}{2}}=mn^{\frac{m}{2}}.
\end{align*}
If $n>1$ is odd then we write $n=2n'+1$ for some $n'\in\N$ and by Lemma \ref{rectmn} we obtain the following:
\begin{align*}
|t_m(n)| & =|t_m(2n'+1)|=\left|\sum_{j=0}^{\lfloor\frac{m-1}{2}\rfloor} {m\choose 2j+1}t_m(n'-j)\right|\leq\sum_{j=0}^{\lfloor\frac{m-1}{2}\rfloor} {m\choose 2j+1}|t_m(n'-j)|\\
& <\sum_{j=0}^{\lfloor\frac{m-1}{2}\rfloor} {m\choose 2j+1}m(n')^{\frac{m}{2}}=m(2n')^{\frac{m}{2}}<mn^{\frac{m}{2}}.
\end{align*}
\end{proof}

Let us observe that the crude estimation using the fact that $|t_{1}(n)|=1$ gives only the equality $t_{m}(n)=O(n^{m})$. The above result shows that there is a lot of cancellation in the sum defining $t_{m}(n)$ and it is quite natural to ask whether the sequence ${\bf t}_{m}$ is bounded or not. Unfortunately, we were unable to answer this question in general but we believe that the following is true.

\begin{conj}\label{unbound}
For each $m\in\N_{\geq 2}$ we have $\limsup_{n\to +\infty}t_{m}(n)=+\infty$ and $\liminf_{n\to +\infty}t_{m}(n)=-\infty$.
\end{conj}

The next result shows that if the sequence ${\bf t}_{m}$ is unbounded on one side then it is unbounded on both sides. 

\begin{lem}
 Let $m \geq 2$. If $\displaystyle\limsup_{n \to \infty} |t_m(n)| = +\infty$ then $\displaystyle\limsup_{n \to \infty} t_m(n) = +\infty$ and $\displaystyle\liminf_{n \to \infty} t_m(n) = -\infty$.
\end{lem}
\begin{proof}
Suppose that $\displaystyle\limsup_{n \to \infty} t_m(n) = +\infty$ and $t_m(n) \geq -C$ for some positive constant $C$. We have that
\begin{equation*}
 C \geq -t_m(2n+1) =  \sum_{i=0}^{[\frac{m-1}{2}]} \binom{m}{2i+1} t_m(n-i)
\end{equation*}
Therefore,
\begin{equation*}
 C+C\sum_{i=0}^{[\frac{m-1}{2}]} \binom{m}{2i+1} \geq  \sum_{i=0}^{[\frac{m-1}{2}]} \binom{m}{2i+1} (t_m(n-i)+C) \geq m (t_m(n)+C)
\end{equation*}

The number $C+C\displaystyle\sum_{i=0}^{[\frac{m-1}{2}]} \binom{m}{2i+1}$ is a constant independent of $n$.
From our assumption $\displaystyle\limsup_{n \to \infty} m (t_m(n)+C) = + \infty$ so we get a contradiction.

One can prove our lemma in the remaining case $\displaystyle\liminf_{n \to \infty} t_m(n) = -\infty$ and $t_m(n) < C$ for some positive constant $C$, by replacing $t_m(n)$ by $-t_m(n)$.
\end{proof}

Using the expression for $\nu_{2}(t_{2^{k}}(n))$ presented in Theorem \ref{v2powerof2} and the above result, we immediately get

\begin{thm}
The Conjecture \ref{unbound} is true for $m=2^{k}$.
\end{thm}
\begin{proof}
Apply Theorem \ref{v2powerof2} in the case of $m=2^{k}$ and Theorem \ref{v2for3} in the case of $m=3$.
\end{proof}

In the case of $m=2, 3$ we can give more precise result. Let
\begin{align*}
&\op{Max}_{m}(k)=\op{max}\{t_{m}(n):\;n\in [0,2^{k}]\},\\
&\op{Min}_{m}(k)=\op{min}\{t_{m}(n):\;n\in [0,2^{k}]\}.
\end{align*}

\begin{thm}\label{maxmin23}
Let $k\in\N_3$. We have the following equalities:
\begin{equation*}
\begin{array}{lllll}
\op{Max}_{2}(k) & = &2^{2\lfloor\frac{k}{2}\rfloor}           & = &t_{2}(2^{2\lfloor\frac{k}{2}\rfloor}-1),\\
\op{Min}_{2}(k) & = & -2^{2\left\lfloor\frac{k-1}{2}\right\rfloor +1} & = &t_2\left(2^{2\left\lfloor\frac{k-1}{2}\right\rfloor +1}-1\right).
\end{array}
\end{equation*}

Moreover, for $m=3$ and $k\in\N$ we have
\begin{equation*}
\begin{array}{lll}
  \op{Max}_{3}(2k)   & = & 2^{3k},                   \\
  \op{Max}_{3}(2k+1) & = & \frac{15}{7}(2^{3k}-1)+[k=0],   \\
  \op{Min}_{3}(2k)   & = & -\frac{3}{7}(2^{3k+1}+5),           \\
  \op{Min}_{3}(2k+1) & = & -3\cdot 2^{3k},
\end{array}
\end{equation*}
and
\begin{align*}
&\op{Max}_{3}(k)=t_{3}\left(2^{k}-\frac{1}{2}(1+(-1)^{k})\right),  k\geq 1,\\
&\op{Min}_{3}(k)=t_{3}\left(2^{k}-\frac{1}{2}(1-(-1)^{k})\right).
\end{align*}
\end{thm}
\begin{proof}
We start with the case of $m=2$. First, let us observe that $t_2(2^k-1)=(-2)^k$ for each $k\in\N_+$. We will prove by induction on $k\in\N_3$ the following statement:
\begin{center}
If $n\in\{0,...,2^k\}\bs\{2^{k-1}-1, 2^k-1\}$ then $|t_2(n)|<2^{k-1}$ and $\sgn t_2(2^k-2)=\sgn t_2(2^k) = -\sgn t_2(2^k-1)$.
\end{center}
Clearly, our statement is true for $k=3$. Let us assume that the statement holds for some $k\in\N_3$. We will show that it holds for $k+1$. If $n\leq 2^k$ and $n\neq 2^k-1$ then obviously $|t_2(n)|<2^k$. Hence it suffices to prove the statement for $n>2^k$. Let us consider the case $n=2l$. If $l\not\in\{2^k-1, 2^k\}$ then $|t_2(n)|\leq |t_2(l)|+|t_2(l-1)|<2^k$, since $|t_2(l)|$ and $|t_2(l-1)|$ are less than $2^{k-1}$. If $l=2^k-1$ then we use the facts that $0<|t_2(2^k-2)|<2^{k-1}<|t_2(2^k-1)|$ and $\sgn t_2(2^k-2)=-\sgn t_2(2^k-1)$ to obtain $|t_2(2^{k+1}-2)|=|t_2(2^k-1)+t_2(2^k-2)|<|t_2(2^k-1)|=2^k$ and $\sgn t_2(2^{k+1}-2)=\sgn t_2(2^k-1)=\sgn \left(-\frac{1}{2}t_2(2^{k+1}-1)\right)=-\sgn t_2(2^{k+1}-1)$. Analogously we prove that $|t_2(2^{k+1})|<2^k$ and $\sgn t_2(2^{k+1})=-\sgn t_2(2^{k+1}-1)$. We are left with the case $n=2l+1$. If $n=2l+1\neq 2^{k+1}-1$ then $l\neq 2^k-1$. Since $l\leq 2^k$, by induction hypothesis we have $|t_2(n)|=2|t_2(l)|<2^{k+1}$.

Summing up our discussion, if $k\in\N_3$ and $n\in\{0,...,2^k\}$ then $t_2(n)$ takes on extremal values for $n\in\{2^{k-1}-1, 2^k-1\}$.

\bigskip

In order to get expressions for $\op{Max}_{3}(k)$ and $\op{Min}_{3}(k)$ we introduce some notation. Let
$F_{1}(k)$ (respectively $F_{2}(k)$) be the right side of the expression for $\op{Max}_{3}(k)$ (respectively $\op{Min}_{3}(k)$) from the statement of our theorem. In the sequel we will need the following fact: if $k\in\N_{+}$ and $n\in \{0,\ldots,2^{k}-2\}$, then
\begin{equation}\label{needed}
\frac{1}{2}F_{2}(k)<t_{3}(n)<\frac{1}{2}F_{1}(k).
\end{equation}
One can easily check that $2F_{1}(k)<F_{1}(k+1)$, $2F_{2}(k)>F_{2}(k+1)$, $-2F_{2}(k)<F_{1}(k+1)$ and $-2F_{1}(k)>F_{2}(k+1)$ for $k\in\N$.

The proof follows by simple induction on $k$. Indeed, the statement is true for $k=0, 1, 2, 3$. Suppose that our inequalities hold for some $k\geq 4$ and take $n\in \{0,\ldots,2^{k+1}-2\}$. We consider two cases: $n$ even and $n$ odd.

If $n=2n'$ then $n'\in \{0,\ldots,2^{k}-1\}$. If $n'=2^{k}-1$ then a simple computation reveals that
$$
t_{3}(n)=t_{3}(2n')=t_{3}(2^{k+1}-2)=\frac{1}{2}(1+(-1)^{k})2^{\frac{3k}{2}}
$$
and that the inequalities (\ref{needed}) are true in this case. If $n'\in \{0,\ldots,2^{k}-2 \}$ then applying the recurrence relations and the induction hypothesis, we get
$$
t_{3}(n)=t_{3}(2n')=t_{3}(n')+3t_{3}(n'-1)<\frac{1}{2}F_{1}(k)+\frac{3}{2}F_{1}(k)=2F_{1}(k)<F_{1}(k+1).
$$
Similarly,
$$
t_{3}(n)=t_{3}(n')+3t_{3}(n'-1)>\frac{1}{2}F_{2}(k)+\frac{3}{2}F_{2}(k)=2F_{2}(k)>F_{2}(k+1).
$$

If $n$ is odd then $n=2n'+1$ for some $n'\in \{0,\ldots,2^{k}-2\}$ and then
$$
t_{3}(n)=t_{3}(2n'+1)=-3t_{3}(n')-t_{3}(n'-1)<-\frac{3}{2}F_{2}(k)-\frac{1}{2}F_{2}(k)=-2F_{2}(k)<F_{1}(k+1).
$$
Similarly
$$
t_{3}(n)=-3t_{3}(n')-t_{3}(n'-1)>-\frac{3}{2}F_{1}(k)-\frac{1}{2}F_{1}(k)=-2F_{1}(k)>F_{2}(k+1).
$$

In order to finish the proof it is enough to observe the equalities
$$
F_{1}(k)=t_{3}\left(2^{k}-\frac{1}{2}(1+(-1)^{k})\right),\quad F_{2}(k)=t_{3}\left(2^{k}-\frac{1}{2}(1-(-1)^{k})\right)
$$
and thus $F_{1}(k)=\op{Max}\{t_{3}(n):\;n\in \{0,\ldots,2^{k}\}\}$ and $F_{2}(k)=\op{Min}\{t_{3}(n):\;n\in \{0,\ldots,2^{k}\}\}$. Our result follows.
\end{proof}

\subsection{Vanishing of $t_{3}(n)$ and more properties of ${\bf t}_{2}$}

In Theorem \ref{v2powerof2}  we have found the explicit $2$-adic valuation of $t_{2^{k}}(n)$. Because the computed numbers are finite for each $n\in\N$, as a consequence we get that the equation $t_{2^{k}}(n)=0$ has no solution for each $k$. Because $t_{3}(2)=0$ it is quite natural to ask about a precise description of the sequence $(a_{k})_{k\in\N_{+}}$ defined by the property
$$
t_{3}(n)=0\Longleftrightarrow n=a_{k}\;\mbox{for some}\;k\in\N_{+}.
$$
Although a description is given in Theorem \ref{v2for3} in terms of the expansion of the integer $n$ in base 4 with digits from the set $\{0,1,3,6\}$, we present a different one in terms of recurrence sequences. More precisely, we have the following.

\begin{thm}\label{23zero}
We have $t_{2^k}(n)\neq 0$ for all $k,n\in\N$. Moreover, $t_{3}(n)=0\Leftrightarrow n=a_{k}$ for some $k\in\N_{+}$,
where the sequence $(a_{k})_{k\in\N_+}$ satisfies the recurrence relation: $a_{1}=2$ and
\begin{equation*}
a_{2k}=4a_{k}+3,\quad a_{2k+1}=4a_{k}+6
\end{equation*}
for $k\geq 1$
\end{thm}
\begin{proof}
The first part of our theorem is very easy. Indeed, we have $\nu_2(t_{2^k}(n))=\nu_2\left(n+2^k-1\choose 2^k-1\right)$. Since ${n+2^k-1\choose 2^k-1}\neq 0$, thus $t_{2^k}(n)\neq 0$.

In order to prove the second part of our theorem we use the results obtained in Lemma \ref{8times} and Proposition \ref{v2for3}, namely
\begin{align*}
t_{3}(4n)  &\equiv 1\pmod{2},\\
t_{3}(4n+1)&\equiv 1\pmod{2}.
\end{align*}
and
\begin{equation}\label{reductiont3}
t_{3}(4n+3)=8t_{3}(n),\quad\quad t_{3}(4n+6)=8t_{3}(n).
\end{equation}
In particular $t_{3}(n)\neq 0$ for $n\equiv 0,1\pmod{4}$.

The equalities in (\ref{reductiont3}) show that if $A_{3}=\{n\in\N_{+}:\;t_{3}(n)=0\}$ then
$$
n\in A_{3}\Longleftrightarrow 4n+3\in A_{3} \;\mbox{and}\;4n+6\in A_{3}.
$$
We have $t_{3}(0)\neq 0$, $t_{3}(1)\neq 0$ and $t_{3}(2)=0$ and thus
$$
A_{3}=\{2,11,14,47,50,59,72,191, 194, 203,\ldots \}.
$$

We prove that $A_{3}=A_{3}'$, where $A_{3}':=\{a_{1}, a_{2}, a_{3}, \ldots\}$, where $a_{0}=-1$ and for $k\geq 1$ we have
$$
a_{2k}=4a_{k}+3,\quad a_{2k+1}=4a_{k}+6.
$$
From the equalities given in (\ref{reductiont3}) and the fact that $a_{1}=2$, we get $a_{k}\in A_{3}$. Let us suppose that $A_{3}\neq \{a_{1},a_{2},\ldots\}$ and let $b$ be the smallest element of $A_{3}$ such that $b\neq a_{k}$ for $k\in\N$. It is clear that $b>10$. However, this implies that $b\equiv 2\;\mbox{or}\;3\mod{4}$. If $b=4n+2$ then $0=t_{3}(b)=t_{3}(4n+2)=8t_{3}(n-1)$ and thus, from the minimality of $b$, we get $n-1=a_{k}\in A_{3}'$ for some $k\in\N_{+}$. We then have $b=4(n-1)+6=4a_{k}+6=a_{2k+1}\in A_{3}'$ -- a contradiction. Similarly, if $b=4n+3$ then $0=t_{3}(b)=t_{3}(4n+3)=8t_{3}(n)$ and thus we get $n=a_{k}\in A_{3}'$ for some $k\in\N_{+}$. Then $b=4n+3=4a_{k}+3=a_{2k}\in A_{3}'$ -- a contradiction.
\end{proof}

The above result has an interesting consequence.

\begin{cor}\label{reduc}
Let us consider the sequence of polynomials ${\bf f}(t)=(f_n(t))_{n\in\N}$ defined as the coefficients in the power series expansion of the series $F_t(x)=F(x)^t$, where $F(x)=\prod_{n=0}^{\infty} \left(1-x^{2^n}\right)$. If $n=a_k$, where the sequence $(a_k)_{k\in\N_+}$ is defined in Theorem \ref{23zero}, then the polynomial $\frac{f_n(t)}{t}$ is reducible as a polynomial in $\Q[t]$.
\end{cor}
\begin{proof}
If $n=a_k$, then $f_n(3)=f_{a_k}(3)=t_3(a_k)=0$ for each $k\in\N_+$.
\end{proof}

We expect that the vanishing of certain terms of the sequence ${\bf t}_{3}$ is an exception and believe that the following is true

\begin{conj}
If $m\in\N_{\geq 4}$ then the equation $t_{m}(n)=0$ has no solution in positive integers.
\end{conj}

\bigskip

Now we turn our attention to the behaviour of the sequence ${\bf t}_{2}$ and prove that its values cover the set $\Z\setminus\{0\}$. Before we present our result let us also note that the sequence ${\bf t}_{2}$ is known as sequence $A106407$ in \cite{OEIS} and it is closely related to the sequence of the Stern polynomials $(B_{n}(t))_{n\in\N}$ defined by the recurrence relation:
$$
B_{0}(t)=0,\; B_{1}(t)=1,\quad B_{2n}(t)=tB_{n}(t),\quad B_{2n+1}(t)=B_{n}(t)+B_{n+1}(t).
$$
The Stern polynomials were introduced by Kla\v{z}ar, Milutinovi\'{c} and Petr in \cite{KMP}. Arithmetic properties of these polynomials were investigated in \cite{Ul1, Ul2} and also in \cite{Gaw}. The connection of ${\bf t}_{2}$ with the Stern polynomials is clear: we have
$$
t_{2}(n)=B_{n+1}(-2).
$$
 This is interesting to note that $B_{n}(2)=n$ and $\{B_{n}(1):\;n\in\N\}=\N$. Moreover, the Stern sequence, i.e., the sequence $(B_{n}(1))_{n\in\N_{+}}$, can be also used to enumerate the positive rational numbers. More precisely, the values of the sequence $(B_{n+1}(1)/B_{n}(1))_{n\in\N_{+}}$ cover $\Q_{+}$ without repetitions.

We will show that $B_{n+1}(-2)=t_{2}(n)$ has a similar property.

First, we show that if $t_{2}(n)=k$ has a solution then there are infinitely many solutions.

\begin{lem}\label{kinf}
Let $m$ be a positive integer $m \geq 3$. Then the following equalities hold
\begin{equation*}
\begin{array}{llllll}
   t_2(8n+4) & = & t_2(2^mn+4), & t_2(8n+6) &=& t_2(2^mn+6), \\
    t_2(8n) &= &t_2(2^mn+2^m-8), & t_2(8n+2) &=& t_2(2^mn+2^m-6),
\end{array}
\end{equation*}

for each positive integer $n$.
\end{lem}
\begin{proof}
We prove the first equality. We have
\begin{align*}
t_2(2^mn+4) &= t_2(2^{m-1}n+2)+t_2(2^{m-1}n+1) \\
	    &= t_2(2^{m-2}n+1) + t_2(2^{m-2}n)-2t_2(2^{m-2}n) \\
	    &= -2t_2(2^{m-3}n) -t_2(2^{m-3}n)-t_2(2^{m-3}n-1) \\
	    &= -3t_2(2^{m-3}n)-t_2(2^{m-3}n-1) \\
	    &= -3(t_2(2^{m-4}n+t_2(2^{m-4}n-1))+2t_2(2^{m-4}n-1) \\
	    &= -3t_2(2^{m-4}n)-t_2(2^{m-4}n-1) \\
	    &= \ldots \\
	    &= -3t_2(n)-t_2(n-1).\\
\end{align*}
Thus the value of $t_2(2^mn+4)$ does not depend on $m$, and our equality holds.
One can prove the other equalities in the same manner.
\end{proof}

\begin{thm}\label{valuesoft2}
For each $k\in\N_{+}$ the equation $t_{2}(n)=k$ has infinitely many solutions in positive integers.
\end{thm}
\begin{proof}
 Let us consider the sequence of rational numbers $\left(\frac{t_2(n+1)}{t_2(n)}\right)_{n\in\N}$.
We prove that for each pair of co-prime positive integers $x,y$ where $x$ is odd and $y$ is even
 one of the fractions $\frac{x}{y}, \frac{y}{x}, -\frac{y}{x}, - \frac{x}{y}$ is in our sequence.
It is a generalisation of the well-known property of Stern diatomic sequence observed by Calkin and Wilf \cite{CalWilf}.
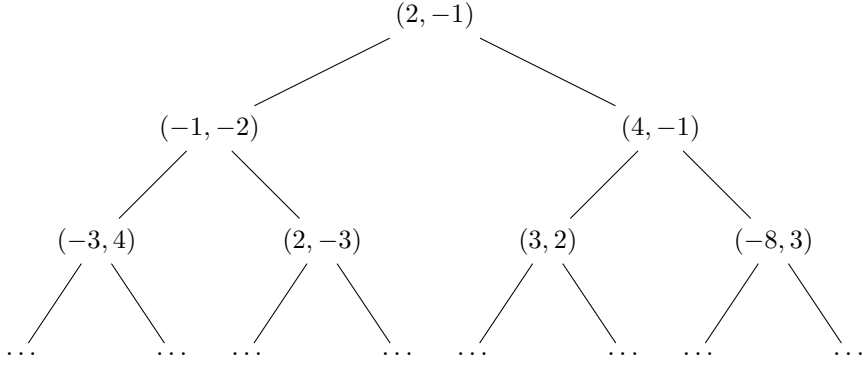
\begin{figure}
\begin{tikzpicture}[level/.style={sibling distance=60mm/#1}]
\node  (a1){$(2,-1)$}
  child {node   (a2) {$(-1,-2) $}
    child{
      node  (a4) {$(-3,4) $}
      child{
	node  (a8) {$\ldots$}
      }
      child{
	node  (a9) {$\ldots$}
      }
    }
    child{
      node (a5) {$(2,-3) $}
      child{
	node (a10) {$\ldots$}
      }
      child{
	node (a11) {$\ldots$}
      }
    }
  }
  child {node (a3) {$(4,-1)$}
    child{
      node  (a6) {$(3,2)$}
      child{
	node (a12) {$\ldots$}
      }
      child{
	node (a13) {$\ldots$}
      }
    }
    child{
      node (a7) {$(-8,3)$}
      child{
	node (a14) {$\ldots$}
      }
      child{
	node (a15) {$\ldots$}
      }
    }
  };
\end{tikzpicture}
\caption{Binary tree rooted in $(2,-1)$}
\end{figure}

Let us consider the following four infinite binary trees of pairs of integers. In the root we put one of the pairs
$(2, -1)$, $(-2, 1)$, $(1,-2)$, $(-1,2)$. In the left child of $(x,y)$ we put $(x+y,-2y)$ and in the right
child we put $(-2x,x+y)$. We will prove that each pair of co-prime non-zero integers such that one of them
is even is in exactly one of our trees.

Suppose that there is a pair of co-prime non-zero integers $(a,b)$ such that one of them is even which is not
in one of our trees.   Let us choose such pair $(a,b)$
with smallest $|a|+|b|$ and in case of a tie with smallest $|a+b|$.
Without loss of generality $(a,b) = (2x,y)$ (when $b$ is even we proceed in the same way).
Let us consider the pair $(-x,x+y)$. Of course $\gcd(-x,x+y)= \gcd(x,y) = 1$, moreover exactly one of the numbers $-x,x+y$
is even. We have that $|-x|+|x+y| \leq |x|+|x|+|y| = |2x|+|y|$ and equality holds if and only if $x$ and $y$ have
the same sign. In that case $|2x+y| > |-x + (x+y) | = |y|$. So from our assumptions either $(-x,x+y)$ is in one of our
trees or $x+y=0$. If $x+y=0$ then $x= \pm 1, y = \mp 1$ and $(2x,y)$ is one of the roots -- a contradiction.
So $(-x,x+y)$ is in one of our trees but its right child is $(2x,y)$, again a contradiction.

Let us observe that the tree with the root $(-a,-b)$ can be obtained from the tree with the root $(a,b)$ by multiplying
all numbers in tree by $-1$. Moreover, the tree with root $(a,b)$ can be obtained from the tree with root $(b,a)$
by swapping the left and right child of each node and swapping the numbers in each pair. We can see that for each valid pair
$(x,y)$ at least one of the pairs $(x,y),(-x,-y),(y,x),(-y,-x)$ is in the tree rooted in $(-2,1)$.
Moreover, from our recurrence relation we get that when we read nodes of that tree row by row from left to right
then we get the sequence of pairs $((t_2(n+1),t_2(n)))_{n=0}^{\infty}$.

Suppose that some odd integer $-(2n+1)$, for $n \neq 0,1$, is not contained in our sequence. Let us look at the pair $(-2,2n+1)$;
from our observations we get that one of the pairs $(-2,2n+1), (2,-(2n+1)), (2n+1,-2), (-(2n+1),2)$ is contained in the tree rooted at $(-2,1)$.
We know that it has to be $(-2, 2n+1)$ or $(2n+1,-2)$ because $-(2n+1)$ is not a member of our sequence.
Let us assume that $(-2,2n+1)$ is in our tree. Then its parent
is $(1,2n)$. The parent of $(1,2n)$ is $(n+1, -n)$. The second child of $(n+1, -n)$ is $(-2(n+1), 1)$ and one of its children is
$(-2n-1, -2)$. We get that $-(2n+1)$ is one of the terms of our sequence -- a contradiction. The second case when $(2n+1,-2)$ is contained
in our tree can treated in the same manner.

Therefore for each odd integer $k$ we can find an $n$ such that $t_2(n) = k$. Using now Lemma \ref{kinf}, we get the statement of our theorem for odd integers $k$. As every even number can be written in the form $(-2)^e(2n+1)$, our theorem holds for even integers as well.

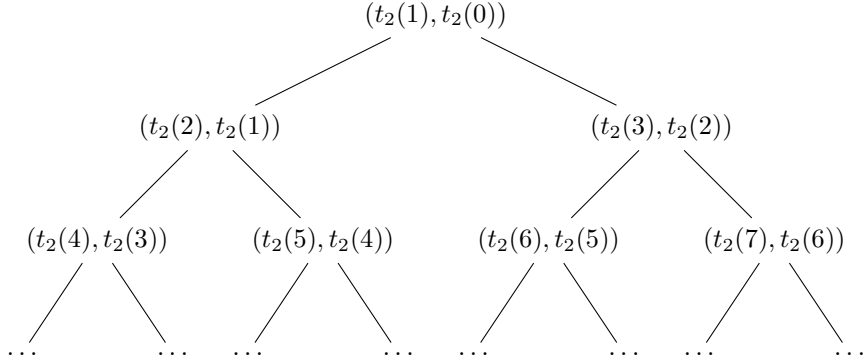
\begin{figure}[!h]
\begin{tikzpicture}[level/.style={sibling distance=60mm/#1}]
\node  (a1){$(t_2(1),t_2(0))$}
  child {node   (a2) {$(t_2(2),t_2(1))$}
    child{
      node  (a4) {$(t_2(4),t_2(3))$}
      child{
	node  (a8) {$\ldots$}
      }
      child{
	node  (a9) {$\ldots$}
      }
    }
    child{
      node (a5) {$(t_2(5),t_2(4)) $}
      child{
	node (a10) {$\ldots$}
      }
      child{
	node (a11) {$\ldots$}
      }
    }
  }
  child {node (a3) {$(t_2(3),t_2(2))$}
    child{
      node  (a6) {$(t_2(6),t_2(5))$}
      child{
	node (a12) {$\ldots$}
      }
      child{
	node (a13) {$\ldots$}
      }
    }
    child{
      node (a7) {$(t_2(7),t_2(6))$}
      child{
	node (a14) {$\ldots$}
      }
      child{
	node (a15) {$\ldots$}
      }
    }
  };
\end{tikzpicture}
\caption{The above tree in terms of the sequence $(t_2(n))$}
\end{figure}

\end{proof}

Based on numerical computations, we observed a striking symmetry in the set of values of $t_{2}(n)$. More precisely, we expect that the following is true

\begin{conj}\label{sym}
For each $n\in\N$ and $m=t_{2}(n)$ the following identity holds: $t_{2}(n')=-t_{2}(n)$, where
\begin{equation*}
n'=n+(-1)^{\nu_{2}(m)+\frac{m-2^{\nu_{2}(m)}}{2^{\nu_{2}(m)+1}}}2^{\nu_{2}(m)+1}.
\end{equation*}
\end{conj}
The above conjecture is true as was proved by A. Schinzel. The proof is given in the Appendix.

\bigskip

We expect that Theorem \ref{valuesoft2} is an exception and believe that the following is true:

\begin{conj}
Let $m$ be a positive integer $\geq 3$. Then the set of those $k\in\Z$ such that the equation $t_{m}(n)=k$ has no solution in positive integers is infinite.
\end{conj}

\subsection{Log-concavity of ${\bf t}_{2}$}

 In this subsection we will see that, as in the Prouhet-Thue-Morse sequence, there are no three consecutive terms of the sequence ${\bf t}_{2}$ of the same sign. In order to prove this we will show two interesting inequalities concerning three consecutive terms of the sequence ${\bf t}_{2}$.

\begin{prop}\label{mean}
For each $n\in\N_+$ we have $|t_2(n)|\geq\frac{|t_2(n-1)+t_2(n+1)|}{2}$ with equality for $n$ even.
\end{prop}

\begin{proof}
The statement of the lemma is true for $n=1$. Assume now that the statement is true for some $n$. We will show that it is also true for $2n+1$. We have the following chain of equivalences:
\begin{align*}
 &\quad |t_2(2n+1)|\geq\frac{|t_2(2n)+t_2(2n+2)|}{2}\\
\Longleftrightarrow &\quad 2|t_2(n)|\geq\frac{|t_2(n-1)+2t_2(n)+t_2(n+1)|}{2}\\
\Longleftrightarrow &\quad 4t_2(n)^2\geq \frac{\left(t_2(n-1)+t_2(n+1)\right)^2+4t_2(n)^2+4t_2(n)\left(t_2(n-1)+t_2(n+1)\right)}{4}\\
\Longleftrightarrow &\quad 3t_2(n)^2\geq \frac{\left(t_2(n-1)+t_2(n+1)\right)^2}{4}+t_2(n)\left(t_2(n-1)+t_2(n+1)\right)
\end{align*}
By the induction hypothesis $|t_2(n)|\geq\frac{|t_2(n-1)+t_2(n+1)|}{2}$, thus
\begin{equation}\label{ineq1}
t_2(n)^2\geq\frac{\left(t_2(n-1)+t_2(n+1)\right)^2}{4}
\end{equation}
and
\begin{equation*}
2t_2(n)^2\geq|t_2(n)|\cdot|t_2(n-1)+t_2(n+1)|.
\end{equation*}
The last inequality together with the fact that $t_2(n)\neq 0$ implies that
\begin{equation}\label{ineq3}
2|t_2(n)|\geq |t_2(n-1)+t_2(n+1)|.
\end{equation}
The inequalities (\ref{ineq1}) and (\ref{ineq3}) imply the last inequality in our chain of equivalences, hence the inequality $|t_2(2n+1)|\geq\frac{|t_2(2n)+t_2(2n+2)|}{2}$ is true. Now we prove equality $|t_2(2n)|=\frac{|t_2(2n-1)+t_2(2n+1)|}{2}$. We have
\begin{align*}
|t_2(2n)|=\frac{|t_2(2n-1)+t_2(2n+1)|}{2}\Longleftrightarrow |t_2(n)+t_2(n-1)|=\frac{|-2t_2(n-1)-2t_2(n)|}{2}
\end{align*}
and as the second equality is true, the first one holds.
\end{proof}

We apply the above result in order to get the following

\begin{thm}
The sequence $(t_2(n))_{n\in\N}$ is log-concave, i.e., for each integer $n \geq 1$ the following inequality holds
\begin{equation}\label{concave}
t_2(n)^2 > t_2(n-1)t_2(n+1).
\end{equation}
The above inequality is optimal in the sense that $t_2(n)^2 = t_2(n-1)t_2(n+1)+1$ for infinitely many positive integers $n$
\end{thm}

\begin{proof}
We present two different proofs of the inequality (\ref{concave}).

{\it First proof.} The inequality (\ref{concave}) holds for $n=1$. Assume now that (\ref{concave}) is true for some $n\in\N_+$. We will show that (\ref{concave}) holds for $2n+1$. We have the following chain of equivalences:
\begin{align*}
 &\quad t_2(2n+1)^2 >t_2(2n)t_2(2n+2)\\
\Longleftrightarrow &\quad 4t_2(n)^2>\left(t_2(n-1)+t_2(n)\right)\left(t_2(n)+t_2(n+1)\right)\\
\Longleftrightarrow & \quad 4t_2(n)^2>t_2(n-1)t_2(n+1)+t_2(n)^2+t_2(n)\left(t_2(n)+t_2(n+1)\right)\\
\Longleftrightarrow & \quad 3t_2(n)^2>t_2(n-1)t_2(n+1)+t_2(n)\left(t_2(n)+t_2(n+1)\right).
\end{align*}
By the induction hypothesis $t_2(n)^2 > t_2(n-1)t_2(n+1)$ and by Proposition \ref{mean} we have $2t_2(n)^2\geq |t_2(n)|\cdot |t_2(n)+t_2(n+1)|\geq t_2(n)\left(t_2(n)+t_2(n+1)\right)$. These two inequalities imply the last inequality in the above chain of equivalences. We thus obtain the inequality $t_2(2n+1)^2>t_2(2n)t_2(2n+2)$. Now we prove the inequality $t_2(2n+1)^2>t_2(2n)t_2(2n+2)$. We have the following equivalences:
\begin{align*}
&\quad  t_2(2n)^2 >t_2(2n-1)t_2(2n+1)\\
\Longleftrightarrow &\quad \left(t_2(n)+t_2(n-1)\right)^2>4t_2(n-1)t_2(n)\\
\Longleftrightarrow &\quad \left(t_2(n)-t_2(n-1)\right)^2>0.
\end{align*}
The last inequality holds since $t_2(n)$ and $t_2(n-1)$ have different parity. Hence we have $t_2(2n)^2>t_2(2n-1)t_2(2n+1)$.

It remains to prove that $t_2(n)^2 = t_2(n-1)t_2(n+1)+1$ for infinitely many positive integers $n$. We will show that this equality holds for $n=2^k-4$, where $k$ is any positive integer $\geq 3$. By simple induction we prove that $t_2(2^k-2)=\frac{1}{3}\left(1-(-2)^k\right)$ and $t_2(2^k-1)=(-2)^k$ for any $k\in\N_+$. Finally we compute for $k\geq 3$:
\begin{align*}
t_2(2^k-4)= & t_2(2^{k-1}-2)+t_2(2^{k-1}-3)=t_2(2^{k-2}-1)+t_2(2^{k-2}-2)-2t_2(2^{k-2}-2)\\
= & t_2(2^{k-2}-1)-t_2(2^{k-2}-2)=\frac{1}{3}\left((-2)^k-1)\right),\\
t_2(2^k-5)= & -2t_2(2^{k-1}-3)=4t_2(2^{k-2}-2)=\frac{1}{3}\left(4-(-2)^k\right)=1-t_2(2^k-4),\\
t_2(2^k-3)= & -2t_2(2^{k-1}-2)=\frac{1}{3}\left(-2-(-2)^k\right)=-1-t_2(2^k-4).
\end{align*}
We thus obtain $t_2(2^k-5)t_2(2^k-3)=t_2(2^k-4)^2-1$ and our theorem follows.
\bigskip

{\it Second proof of the inequality (\ref{concave}).} Let us define $a(n) = t_2(n+1)$. The sequence $a(n)$ satisfies the following recurrence relations $a(1) = 1$, $a(2n) = -2a(n)$, and  $a(2n+1) = a(n+1) + a(n)$. It is enough to prove our inequality for the sequence $a(n)$.

Let $n=2^k(2l+1)$ and $l \geq 1$. Applying the recurrence relations $k+1$ times we get
 \begin{align*}
  a(n) &= (-2)^k (a(l) + a(l+1)), \\
  a(n-1) &= \left(b(k)-2 \right)a(l) + b(k) a(l+1), \\
  a(n+1) &= \left(b(k)-2 \right)a(l+1) + b(k)a(l),
 \end{align*}
 where $b(k) = -1/3 \cdot \left((-2)^k-1\right)$. We compute:
 \begin{align*}
  a(n)^2&-a(n-1)a(n+1) \\
        &= (2^{2k}-b(k)(b(k)-2))(a(l)^2+a(l+1)^2) + (2^{2k+1} - b(k)^2-(b(k)-2)^2) a(l)a(l+1)
 \end{align*}
and observe that it is enough to prove that $ 2(2^{2k}-b(k)(b(k)-2)) \geq \left|(2^{2k+1} - b(k)^2-(b(k)-2)^2)\right|$.
It is not hard to see that the last term is positive when $k>0$ because $|b(k)| \leq \frac{1}{3} (2^k+1)$. Let us compute the difference
\begin{equation*}
  2\left(2^{2k}-b(k)\left(b(k)-2\right)\right) - \left(2^{2k+1} - b(k)^2-\left(b(k)-2\right)^2\right) = 4.
 \end{equation*}
 When $k=0$ we get $a(n)^2-a(n-1)a(n+1) = (a(l+1)-a(l))^2 > 0$. It cannot be zero because $a(l)$ and $a(l+1)$ have different parity.
Therefore our inequality holds.

Finally, let $n=2^k$. It is easy to see that $(a(n-1),a(n),a(n+1)) = (b(k),(-2)^k,b(k)-2)$. Moreover,
\begin{align*}
 a(n)^2&-a(n+1)a(n-1)\\
       &= 2^{2k}-(b(k)-2)b(k) \geq 2^{2k} - 1/9 (2^k+1)(2^k+3) = 1/9(2^{2k+3}-2^{k+2}-3) > 0
\end{align*}
and our theorem follows.
\end{proof}

We are ready to prove that none of three consecutive terms of the sequence $(t_2(n))_{n\in\N}$ have the same sign. Let us note that the same property holds for ${\bf t}_{1}$ - the Prouhet-Thue-Morse sequence.

\begin{thm}
For any positive integer $n$ the numbers $t_2(n-1)$, $t_2(n)$, $t_2(n+1)$ do not have the same sign.
\end{thm}

\begin{proof}
The statement of our theorem is true for $n=1$. Assume that $t_2(n-1)$, $t_2(n)$, $t_2(n+1)$ do not have the same sign and consider the numbers $t_2(2n-1)$, $t_2(2n)$, $t_2(2n+1)$. If $t_2(2n-1)$ and $t_2(2n+1)$ have the same sign then $t_2(n-1)$ and $t_2(n)$ have the same sign, since $t_2(2n-1)=-2t_2(n-1)$ and $t_2(2n+1)=-2t_2(n)$. However, the sign of $t_2(2n-1)$ and $t_2(2n+1)$ is different from the sign of $t_2(n-1)$ and $t_2(n+1)$ while the sign of $t_2(2n)=t_2(n-1)+t_2(n)$ is the same as the sign of $t_2(n-1)$ and $t_2(n+1)$. Consider now the numbers $t_2(2n)$, $t_2(2n+1)$, $t_2(2n+2)$ and suppose that they have the same sign. Then $\sgn t_2(n)=\sgn \left(-\frac{1}{2}t_2(2n+1)\right)=-\sgn t_2(2n)$. Since $t_2(2n)=t_2(n-1)+t_2(n)$ and $\sgn t_2(2n)=\sgn t_2(2n+1)$, thus $\sgn t_2(n-1)=-\sgn t_2(n)$ and $|t_2(n-1)|>|t_2(n)|$. Analogously we conclude that $\sgn t_2(n+1)=-\sgn t_2(n)$ and $|t_2(n+1)|>|t_2(n)|$. Thus $t_2(n)^2<t_2(n-1)t_2(n+1)$, which contradicts the inequality (\ref{concave}).
\end{proof}

\section{Arithmetic properties of the sequence $(f(n,t))_{n\in\N}$ with $t\in\Z_{<0}$}\label{Section4}

In this section we consider the sequence $(f_{n}(t))_{n\in\N}$ with a fixed negative integer $t$. We thus put $t=-m$ for $m\in\N_{+}$ and we write
$$
b_{m}(n):=f_{n}(-m).
$$
Moreover, in order to shorten the notation we write
$$
H_{m}(x):=F_{-m}(x)=\sum_{n=0}^{\infty}b_{m}(n)x^{n}.
$$
In particular $b_{1}(n)=b(n)$ is the well known binary partition function introduced by Euler and studied by Churchhouse \cite{Chu}, R{\o}dseth \cite{Rod}, Gupta \cite{Gup} and others. It is sequence $A018819$ in \cite{OEIS}. Also the sequence $(b(2n))_{n\in\N}$ can be found in \cite{OEIS}, namely as sequence $A000123$. It is clear that $b_{m}(n)$ is the convolution of $m$ copies of the sequence $(b(n))_{n\in\N}$. We thus have
\begin{equation*}
b_{m}(n)=\sum_{i_{1}+i_{2}+\ldots+i_{m}=n}\prod_{k=1}^{m}b(i_{k}).
\end{equation*}
From the above expression we easily deduce that the number $b_{m}(n)$ has a natural combinatorial interpretation. Indeed, $b_{m}(n)$ counts the number of representations of the integer $n$ as the sum of powers of $2$, where each summand can have one of $m$ colors.

We start with the proof of the recurrence relations satisfied by the sequence $(b_{m}(n))_{n\in\N}$.

\begin{lem}\label{recforbm}
Let $m$ be a positive integer. Then the sequence $(b_{m}(n))_{n\in\N}$ satisfies $b_{m}(0)=1, b_{m}(1)=m$ and for $n\geq 1$ we have
\begin{align*}
b_{m}(2n)  = &\sum_{j=0}^{m-1}{m\choose j+1}(-1)^{j}b_{m}(2n-j-1)+b_{m}(n),\\
b_{m}(2n+1)= &\sum_{j=0}^{m-1}{m\choose j+1}(-1)^{j}b_{m}(2n-j).
\end{align*}

Moreover, the sequence $(b_{m}(n))_{n\in\N}$ satisfies the following recurrence relations:
\begin{align*}
b_{m}(2n)  = &\sum_{j=0}^{n}{2(n-j)+m-1\choose m-1}b_{m}(j),\\
b_{m}(2n+1)= &\sum_{j=0}^{n}{2(n-j)+m\choose m-1}b_{m}(j).
\end{align*}
\end{lem}
\begin{proof}
The function $H_{m}$ satisfies the functional equation $(1-x)^{m}H_{m}(x)=H_{m}(x^2)$. In consequence we have
\begin{align*}
(1-x)^{m}H_{m}(x)&=\left(\sum_{j=0}^{m}{m\choose j}(-1)^{j}x^{j}\right)\left(\sum_{n=0}^{\infty}b_{m}(n)x^{n}\right)\\
                 &=\sum_{n=0}^{\infty}\left(\sum_{j=0}^{m}{m\choose j}(-1)^{j}b_{m}(n-j)\right)x^{n}=\sum_{n=0}^{\infty}b_{m}(n)x^{2n}.
\end{align*}
Comparing now the coefficients on both sides of the above identity we get two equalities:
$$
\sum_{j=0}^{m}{m\choose j}(-1)^{j}b_{m}(2n+1-j)=0,\quad \sum_{j=0}^{m}{m\choose j}(-1)^{j}b_{m}(2n-j)=b_{m}(n).
$$
From the first equality we get the expression for $b_{m}(2n+1)$. From the second relation we get the expression for $b_{m}(2n)$. Finally, replacing $i$ by $i+1$, we get the relations given in the statement of our theorem.

\bigskip

In order to get the second part of our Lemma we use the same technique. From the functional equation for $H_{m}(x)=(1-x)^{-m}H_{m}(x^2)$ we have
\begin{equation*}
\sum_{n=0}^{\infty}b_{m}(2n)x^{2n}=\frac{1}{2}(H_{m}(x)+H_{m}(-x))=\frac{1}{2}\left(\frac{1}{(1-x)^{m}}+\frac{1}{(1+x)^{m}}\right)H_{m}(x^2).
\end{equation*}
A quick calculation reveals that
\begin{equation*}
\frac{1}{2}\left(\frac{1}{(1-x)^{m}}+\frac{1}{(1+x)^{m}}\right)=\sum_{n=0}^{\infty}\binom{m+2n-1}{2n} x^{2n}
\end{equation*}
and thus (after the substitution $x\mapsto \sqrt{x}$) we have
\begin{align*}
\sum_{n=0}^{\infty}b_{m}(2n)x^{n}&=\left(\sum_{n=0}^{\infty}\binom{m+2n-1}{n} x^{2n}\right)\left(\sum_{n=0}^{\infty}b_{m}(n)x^{n}\right)\\
& =\sum_{n=0}^{\infty}\left(\sum_{j=0}^{n}\binom{m+2(n-j)-1}{2(n-j)}b_{m}(j)\right)x^{n}.
\end{align*}
Comparing now the coefficients on the both sides of the above identity and using the symmetry property of binomial coefficients, we get the first identity given in the statement of our lemma.

Using exactly the same type of reasoning and the identity
\begin{equation*}
\sum_{n=0}^{\infty}b_{m}(2n+1)x^{2n+1}=\frac{1}{2}(H_{m}(x)-H_{m}(-x))=\frac{1}{2}\left(\frac{1}{(1-x)^{m}}-\frac{1}{(1+x)^{m}}\right)H_{m}(x^2)
\end{equation*}
we prove the second identity. We leave the details to the reader.
\end{proof}

\subsection{Some inequalities involving $b_{m}(n)$ for $m=1, 2$}

In the previous section we proved that $t_2(n)^2-t_2(n-1)t_2(n+1)>0$. Using recurrence relations for the numbers $b_m(n)$, we can easily compute the sign of the expression $b_m(n)^2-b_m(n-1)b_m(n+1)$ for $m\in\{1,2\}$.

\begin{prop}\label{equations}
For $n\in\N_+$ the following equalities hold:
\begin{align*}
b_1(2n)^2-b_1(2n-1)b_1(2n+1) & =b_1(2n)b_1(n),\\
b_1(2n-1)^2-b_1(2n-2)b_1(2n) & =-b_1(2n-2)b_1(n),\\
b_2(2n)^2-b_2(2n-1)b_2(2n+1) & =\left(\sum_{j=0}^n b_2(j)\right)^2,\\
b_2(2n-1)^2-b_2(2n-2)b_2(2n) & =\left(\sum_{j=0}^n b_2(j)\right)^2-b_2(2n-2)b_2(n).
\end{align*}
In particular, we have $$(-1)^n\left( b_m(n)^2-b_m(n-1)b_m(n+1)\right)>0$$ for $m\in\{1,2\}$ and each $n\in\N_+$.
\end{prop}

\begin{proof}
We perform direct calculations using the first part of Lemma \ref{recforbm}. If $m=1$ then $b_1(2n)=b_1(2n+1)$ for each $n\in\N$. We thus have
\begin{align*}
 b_1(2n)^2-b_1(2n-1)b_1(2n+1)&=b_1(2n)^2-b_1(2n-2)b_1(2n)\\
                              &=b_1(2n)\left(b_1(2n)-b_1(2n-2)\right)=b_{1}(2n)b_{1}(n).\\
\end{align*}
Similarly,
\begin{align*}
 b_1(2n-1)^2-b_1(2n-2)b_1(2n)&=b_1(2n-2)^2-b_1(2n-2)b_1(2n)\\
                             &=b_1(2n-2)\left(b_1(2n-2)-b_1(2n)\right)=-b_1(2n-2)b_1(n).
\end{align*}

For $m=2$ the computations are more complicated:
\begin{align*}
 b_2(2n)^2&-b_2(2n-1)b_2(2n+1)\\
 &=\left(\sum_{j=0}^n \left(2(n-j)+1\right)b_2(j)\right)^2-\left(\sum_{j=0}^n2(n-j)b_2(j)\right)\left(\sum_{j=0}^n\left(2(n-j)+2\right)b_2(j)\right)\\
 &= \left(\sum_{j=0}^n \left[\left(2(n-j)+1\right)^2-2(n-j)\left( 2(n-j)+2\right)\right]b_2(j)^2\right)+\\
 &+ \left(\sum_{0\leq j<k\leq n} F(2(n-j)+1,2(n-k)+1)b_2(j)b_2(k)\right)\\
 &= \left(\sum_{j=0}^n b_2(j)^2\right)+\left(\sum_{0\leq j<k\leq n} 2b_2(j)b_2(k)\right)=\left(\sum_{j=0}^n b_2(j)\right)^2.
\end{align*}
In the equality between the third and fourth expression in the computation above we applied the identity $F(x,y)=2xy-(x-1)(y+1)-(x+1)(y-1)=2$ for $x=2(n-j)+1$ and $y=2(n-k)+1$.
\end{proof}

\begin{rem}
{\rm The first to prove the first equality in Proposition \ref{equations}, was D. Knuth, as was pointed out by B. Reznick in \cite{Rez}.
}
\end{rem}

We also have the identity:
\begin{align*}
b_2(2n-1)^2-b_2(2n-2)b_2(2n)=\left(\sum_{j=0}^{n-1} b_2(j)\right)^2-b_2(2n-2)b_2(n).
\end{align*}
Indeed, we have the following chain of inequalities:
\begin{align*}
 b_2&(2n-1)^2-b_2(2n-2)b_2(2n)\\
     & =\left(\sum_{j=0}^{n-1} 2(n-j)b_2(j)\right)^2-\left(\sum_{j=0}^{n-1} \left(2(n-j)-1\right)b_2(j)\right)\left(\sum_{j=0}^n
     \left(2(n-j)+1\right)b_2(j)\right)\\
     &= \left(\sum_{j=0}^{n-1} \left[\left(2(n-j)\right)^2-\left( 2(n-j)-1\right)\left( 2(n-j)+1\right)\right]b_2(j)^2\right)+\\
     &+ \left(\sum_{0\leq j<k\leq n-1} G(2(n-j)+1,2(n-k)+1)b_2(j)b_2(k)\right)-\left(\sum_{j=0}^{n-1} \left(2(n-j)-1\right)b_2(j)\right)b_2(n)\\
     &= \left(\sum_{j=0}^{n-1} b_2(j)^2\right)+\left(\sum_{0\leq j<k\leq n-1}2b_2(j)b_2(k)\right)-b_2(2n-2)b_2(n)=\left(\sum_{j=0}^{n-1} b_2(j)\right)^2-b_2(2n-2)b_2(n).
\end{align*}
In the equality between third and fourth expression in the computation above we applied the identity $G(x,y)=2xy-(x-1)(y+1)-(x+1)(y-1)=2$ for $x=2(n-j)$ and $y=2(n-k)$.

It seems that in this case $b_2(2n-1)^2-b_2(2n-2)b_2(2n)<0$ for all $n\in\N_+$, but we were unable to prove this statement.

\subsection{Some congruences involving ${\bf b}_{m}$}

In this subsection we present several congruences involving the sequence ${\bf b}_{m}$ for various values of $m$. We are mainly interested in the congruences $\pmod{2^k}$ for various values of $k$. In particular we give a precise description of the 2-adic valuation of the elements of the sequence ${\bf b}_{2^{k}-1}$  for $k\in\N_{+}$, thereby generalizing the result of Churchhouse.

First we prove a simple lemma concerning the characterization of parity of the number $b_{m}(n)$.

\begin{lem}\label{parity}
 Let $m\in \N_{+}$ be fixed and write $m=2^{k}(2u+1)$ with $k\in\N$. Then:
 \begin{enumerate}
  \item We have $b_{m}(n) \equiv \binom{m}{n}+2^{k+1}\binom{m-2}{n-2} \pmod {2^{k+2}}$ for $m$ even;
  \item We have $b_{m}(n) \equiv \binom{m}{n} \pmod {2}$ for $m$ odd;
  \item For infinitely many $n$ we have $b_{m}(n)\not\equiv 0\pmod{4}$ for $m$ odd.
 \end{enumerate}
\end{lem}

\begin{proof}
To prove the first part, we write $m = 2^k(2u+1)$. Let us observe that for $i=0,1,\ldots,2^{k-1}-1$ we have (by Legendre's formula for the $2$-adic valuation of a factorial) $v_2(\binom{2^k}{2i+1}) = s_2(2i+1)+s_2(2^k-(2i+1))-1 = k$. Thus the following congruence holds
\begin{equation*}
 (1+x)^{2^k}-(1-x)^{2^k} = 2\sum_{i=0}^{2^{k-1}-1} \binom{2^k}{2i+1}x^{2i+1} \equiv 2^{k+1} \sum_{i=0}^{2^{k-1}-1} x^{2i+1} \equiv 2^{k+1}x(1+x)^{2^k-2} \pmod {2^{k+2}}.
\end{equation*}
Using simple calculations we get the following equality:
\begin{align*}
H_{m}(x) &= \prod_{n=0}^{\infty}\left(1-x^{2^{n}}\right)^{-m} \\
	 &\equiv \prod_{n=0}^{\infty} \left( ((1+x^{2^n})^{2^k}+2^{k+1}x^{2^n}(x+1)^{2^n(2^k-2)}) \right)^{-(2u+1)} \\
	 &\equiv \prod_{n=0}^{\infty} \left( (1+x^{2^n})^m+2^{k+1}x^{2^n}(1+x)^{2^n(m-2)} \right)^{-1} \\
	 &\equiv \prod_{n=0}^{\infty} \left( (1+x^{2^n})^m(1+2^{k+1} \frac{x^{2^n}}{(1+x)^{2^{n+1}}} )\right)^{-1} \\
	 &\equiv (1-x)^m \left(1+2^{k+1} \sum_{n=0}^{\infty} \frac{x^{2^n}}{(1+x)^{2^{n+1}}} \right)^{-1} \\
	 &\equiv (1-x)^m \left(1+2^{k+1} \frac{x}{1-x}\right)  \\	
\end{align*}
\begin{align*}
	 &\equiv (1-x)^m +2^{k+1}x(1-x)^{m-1} \\
	 &\equiv (1+x)^m +2^{k+1}x((1+x)^{m-1}+(1+x)^{m-2}) \\
	 &\equiv (1+x)^m+2^{k+1}x^2(1+x)^{m-2} \pmod {2^{k+2}},
\end{align*}
and comparing the coefficients of $x^{n}$ on both sides of the above congruence we get the result.

For the second part; let us observe that
\begin{equation*}
H_m(x) = \prod_{n=0}^{\infty}\left(1-x^{2^{n}}\right)^{-m} \equiv \prod_{n=0}^{\infty}\left(1+x^{2^{n}}\right)^{-m} \equiv (1+x)^m \pmod 2.
\end{equation*}

For the third part; let us observe that $(1+x)\equiv (1-x)+2x \pmod 4$. Using analogous computations as before we get
\begin{equation*}
H_m(x) \equiv (1-x)^m\left(1+2\sum_{n=0}^{\infty} \frac{x^{2^n}}{1+x^{2^n}}\right) \pmod 4.
\end{equation*}
We let $A$ denote the set of those $u \in \mathbb{N}_+$ which end with even number of zeros in binary expansion. We have
\begin{equation*}
\sum_{n=0}^{\infty} \frac{x^{2^n}}{1+x^{2^n}} \equiv \sum_{n=0}^{\infty} \sum_{m=1}^{\infty} x^{m2^n} \equiv \sum_{n \in A} x^n := \xi(x)  \pmod 2.
\end{equation*}
If we had $b_{m}(n)\equiv 0\pmod 4$ for $n$ large enough, then $(1+x)^m\xi(x)$ would be a polynomial in $\mathbb{F}_2[[x]]$. But this is impossible,
since $\xi$ satisfies the formula $\xi+\xi^2 = \frac{x}{1+x}$ (in $\F_2[[x]]$).
\end{proof}

In the sequel we will also need the following simple observations concerning the binomial coefficients modulo 2 and 8.
\begin{lem}\label{binomial8}
 Let $m$ be a positive integer $\geq 2$. Then
\begin{equation*}
{2^{m}-1\choose k}\equiv 1\pmod{2},\quad\mbox{for}\quad k=0,1,\ldots, 2^{m}-1,
\end{equation*}
and
\begin{equation*}
{2^{m}\choose k}\equiv\left\{
\begin{array}{ll}
1 & \mbox{for}\;k=0, 2^{m} \\
4 & \mbox{for}\;k=2^{m-2}, 3\cdot 2^{m-2} \\
6 & \mbox{for}\;k=2^{m-1} \\
0 & \; \mbox{in the remaining cases}
                                     \end{array}
\right\}\pmod{8},\quad\mbox{for}\quad k=0,1,\ldots, 2^{m}.
\end{equation*}
\end{lem}

\begin{proof}
For each $j\in\{1,...,2^m-1\}$ we have $\nu_2(2^m-j)=\nu_2(j)$. Hence $\nu_2((2^m-1)\cdot ...\cdot (2^m-k))=\nu_2(k!)$ for $k\in\{1,...,2^m-1\}$. We thus have
\begin{equation*}
\nu_2\left( 2^m-1\choose k\right)=\nu_2\left(\frac{(2^m-1)\cdot ...\cdot (2^m-k)}{k!}\right)=0,
\end{equation*}
which means that $2^m-1\choose k$ is odd.

Obviously, ${2^m\choose 0}={2^m\choose 2^m}=1$.

We have
\begin{equation*}
{2^m\choose 2^{m-2}}={2^m\choose 3\cdot 2^{m-2}}=\frac{2^m\cdot ...\cdot (3\cdot 2^{m-2}+1)}{2^{m-2}!}=4 \frac{(2^m-1)\cdot ...\cdot (3\cdot 2^{m-2}+1)}{(2^{m-2}-1)!}=4{2^m-1\choose 2^{m-2}-1}.
\end{equation*}
Since $2^m-1\choose 2^{m-2}-1$ is odd, thus ${2^m\choose 2^{m-2}}={2^m\choose 3\cdot 2^{m-2}}\equiv 4\pmod{8}$.

Let us write ${2^m\choose 2^{m-1}}=\prod_{j=1}^{2^{m-1}} \frac{2^{m-1}+j}{j}$. Each $j\in\{1,...,2^{m-1}\}$ can be written in the form $j=2^{k}i$ for some $k,i\in\N$, where $i$ is an odd number. If $k\leq m-4$ then
\begin{equation*}
\frac{2^{m-1}+j}{j}=\frac{2^{m-1}+2^ki}{2^ki}=\frac{2^{m-k-1}+i}{i}\equiv 1\pmod{8}.
\end{equation*}
We thus obtain
\begin{equation*}
{2^m\choose 2^{m-1}}=\prod_{j=1}^{2^{m-1}} \frac{2^{m-1}+j}{j}\equiv\frac{5\cdot 2^{m-3}}{2^{m-3}}\cdot\frac{3\cdot 2^{m-2}}{2^{m-2}}\cdot\frac{7\cdot 2^{m-3}}{3\cdot 2^{m-3}}\cdot\frac{2^m}{2^{m-1}}=5\cdot 3\cdot \frac{7}{3}\cdot 2\equiv 6\pmod{8}.
\end{equation*}
If $\nu_2(k)\leq m-3$ then
\begin{equation*}
{2^m\choose k}=\frac{2^m\cdot ...\cdot (2^m-k+1)}{k!}=\frac{2^m}{k}\cdot\frac{(2^m-1)\cdot ...\cdot (2^m-k+1)}{(k-1)!}=\frac{2^m}{k}\cdot{2^m-1\choose k-1}\equiv 0\pmod{8}
\end{equation*}
and our lemma follows.
\end{proof}

One of the main results of this paper is the following result concerning the computation of the 2-adic valuation of the members of the sequence $(b_{2^{k}-1}(n))_{n\in\N}$ with fixed positive integer $k$. Our next theorem can be seen as a generalization of the identity
\begin{equation*}
\nu_{2}(b_{1}(n))=
\begin{cases}
\frac{1}{2}|t_{n}-2t_{n-1}+t_{n-2}|,& \mbox{if}\;n\geq 2\\
0,& \mbox{if}\;n\in\{0,1\}
\end{cases}
\end{equation*}
obtained by Churchhouse (however, in a slightly different form). Here, $t_{n}$ is the $n$-th term of the Prouhet-Thue-Morse sequence.

\begin{thm}\label{2valpowerof2-1}
Let $k\in\N_{+}$ be given. We then have the following equality
\begin{equation*}
\nu_{2}(b_{2^{k}-1}(2^{k+2}n+i))=\begin{cases}
\begin{array}{ll}
  \nu_{2}(b_{2^{k}-1}(2^{k+2}n))=\nu_{2}(b_{1}(8n)), & \quad\mbox{for}\quad i=0,1,\ldots, 2^{k}-1, \\
  1, & \quad\mbox{for}\quad i=2^{k}, 2^{k}+1,\ldots, 2^{k+1}-1, \\
  2, & \quad\mbox{for}\quad i=2^{k+1}, 2^{k+1}+1,\ldots, 3\cdot 2^{k}-1, \\
  1, & \quad\mbox{for}\quad i=3\cdot 2^{k}, 3\cdot 2^{k}+1,\ldots, 2^{k+2}-1,
\end{array}
\end{cases}
\end{equation*}
for $n\in\N$. In particular $\nu_{2}(b_{2^{k}-1}(n))\in\{0,1,2\}$ and $\nu_{2}(b_{2^k-1}(n))=0$ if and only if $n\leq 2^{k}-1$.
\end{thm}
\begin{proof}
First of all, let us observe that the second part of Lemma \ref{parity} and the first part of Lemma \ref{binomial8} implies that $b_{2^{k}-1}(n)$ is odd for $n\leq 2^{k}-1$ and thus $\nu_{2}(b_{2^{k}-1}(n))=0$ in this case.

Let us observe that from the identity $H_{2^{k}-1}(x)=F_{1}(x)H_{2^{k}}(x)$ we get the identity
\begin{equation}\label{formula}
 b_{2^{k}-1}(n)=\sum_{j=0}^{n}t_{n-j}b_{2^{k}}(j),
\end{equation}
 where $t_{n}=t_{1}(n)=(-1)^{s_{2}(n)}$ is $n$-th term of the Prouhet-Thue-Morse sequence. Now let us observe that from the first part of Lemma \ref{parity} and the second part of Lemma \ref{binomial8} we have
\begin{equation*}
b_{2^{k}}(n)\equiv {2^{k}\choose n}\pmod{8}
\end{equation*}
for $n=0,1,\ldots,2^{k}$ and $b_{2^{k}}(n)\equiv 0\pmod{8}$ for $n>2^{k}$, provided $k\geq 2$ or $n\neq 2$. Moreover,
\begin{equation*}
b_2(2)\equiv {2\choose 2}+4{0\choose 0}=5\pmod{8}.
\end{equation*}
Summing up this discussion we have the following expression for $b_{2^{k}-1}(n)\pmod{8}$, where $k\geq 2$ and $n\geq 2^{k}$:
\begin{align*}
b_{2^{k}-1}(n)&=\sum_{j=0}^{n}t_{n-j}b_{2^{k}}(j)\\
              &=\sum_{j=0}^{2^{k}}t_{n-j}b_{2^{k}}(j)+\sum_{j=2^{k}+1}^{n}t_{n-j}b_{2^{k}}(j)\\
              &\equiv \sum_{j=0}^{2^{k}}t_{n-j}b_{2^{k}}(j)\pmod{8}\\
              &\equiv \sum_{j=0}^{2^{k}}t_{n-j}{2^{k}\choose j}\pmod{8}\\
              &\equiv t_{n}+t_{n-2^{k}}+4t_{n-2^{k-2}}+4t_{n-3\cdot 2^{k-2}}+6t_{n-2^{k-1}}\pmod{8}.
\end{align*}
However, it is clear that $t_{n-2^{k-2}}+t_{n-3\cdot2^{k-2}}\equiv 0\pmod{2}$ and thus we can simplify the above expression and get
\begin{equation*}
b_{2^{k}-1}(n)\equiv t_{n}+t_{n-2^{k}}+6t_{n-2^{k-1}}\pmod{8}
\end{equation*}
for $n\geq 2^{k}$. If $k=1$ and $n\geq 2$ then, analogously, we get
\begin{equation*}
b_{1}(n)\equiv \sum_{j=0}^{2^{k}}t_{n-j}b_{2^{k}}(j)\pmod{8}\equiv t_{n}+5t_{n-2}+2t_{n-1}\pmod{8}
\end{equation*}
and since $t_{n-1}\equiv t_{n-2}\pmod{2}$, we thus conclude that
\begin{equation*}
b_{1}(n)\equiv t_{n}+t_{n-2^{k}}+6t_{n-2^{k-1}}\pmod{8}.
\end{equation*}
Let us put $R_{k}(n)=t_{n}+t_{n-2^{k}}+6t_{n-2^{k-1}}$. Using now the recurrence relations for $t_{n}$, i.e., $t_{2n}=t_{n}, t_{2n+1}=-t_{n}$ we easily deduce the identities
\begin{equation*}
 R_{k}(2n)=R_{k-1}(n),\quad R_{k}(2n+1)=-R_{k-1}(n)
\end{equation*}
for $k\geq 2$. Using a simple induction argument, one can easily obtain the following identities:
\begin{equation}\label{reduction1}
 |R_{k}(2^{k}m+j)|=|R_{1}(2m)|
\end{equation}
for $k\geq 2, m\in\N$ and $j\in\{0,...,2^{k}-1\}$. From the above identity we easily deduce that $R_{k}(n)\not\equiv0\pmod{8}$ for each $n\in\N$ and each $k\geq 1$. If $k=1$ then $R_{1}(n)=t_{n}+6t_{n-1}+t_{n-2}$ and $R_{1}(n)\equiv 0\pmod{8}$ if and only if $t_{n}=t_{n-1}=t_{n-2}$. However, a well known property of the Prouhet-Thue-Morse sequence is that there are no three consecutive terms which are equal. If $k\geq 2$ then our statement about $R_{k}(n)$ is clearly true for $n\leq 2^{k}$. If $n>2^{k}$ then we can write $n=2^{k}m+j$ for some $m\in\N$ and $j\in\{0,1,\ldots,2^{k}-1\}$. Using the reduction (\ref{reduction1}) and the property obtained for $k=1$, we get the result. Summing up our discussion, we have proved that $\nu_{2}(b_{2^{k}-1}(n))\leq 2$ for each $n\in\N$, since $\nu_{2}(b_1(n))=\in\{0,1,2\}$. Moreover, as an immediate consequence of our reasoning we get the equality
\begin{equation*}
 \nu_{2}(b_{2^{k}-1}(2^kn+j))=\nu_{2}(b_{1}(2n))
\end{equation*}
for $j\in\{0,...,2^{k}-1\}$. Using the above identity and the properties of $\nu_{2}(b_{1}(2n))$ we easily get the identities presented in the statement of our theorem.
\end{proof}

It is an interesting question whether we can say something non-trivial about 2-adic valuation of the number $b_{m}(n)$ for $m\neq 2^{k}-1$. In order to do this in the sequel we will need the following lemma concerning the form of the generating function of the subsequence $(b_{m}(2^{k}n+i))_{n\in\N}$, where $k\in\N$ is given and $i\in\{0,\ldots, 2^{k}-1\}$. We have the following

\begin{lem}\label{hik}
Let $m\in\N_{+}$ be fixed. Let $i,k\in\N$ be given and consider the function $H_{i,k}(x)=H_{i,k,m}(x)$ which is the generating function for the sequence $(b_{m}(2^{k}n+i))_{n\in\N}$, where $0\leq i<2^{k}$, i.e., $H_{i,k}(x)=\sum_{n=0}^{\infty}b_{m}(2^{k}n+i)x^{n}$. Then
\begin{equation*}
H_{i,k}(x)=\frac{h_{i,k}(x)}{(1-x)^{mk}}H_{m}(x),
\end{equation*}
where the (double) sequence of polynomials $(h_{i,k}(x))_{k\in\N, 0\leq i<2^{k}}=(h_{i,k,m}(x))_{k\in\N, 0\leq i<2^{k}}$ satisfies $h_{0,0}(x)=1$ and for $k\in\N_{+}$ and $0\leq i<2^{k}$ we have
\begin{equation*}
h_{i,k}(x)=
\begin{cases}
\begin{array}{lll}
  \frac{1}{2}(h_{i,k-1}(\sqrt{x})(1+\sqrt{x})^{mk}+h_{i,k-1}(-\sqrt{x})(1-\sqrt{x})^{mk}), &  & \mbox{for}\quad i=0,1,\ldots,2^{k-1}-1 \\
  \frac{1}{2\sqrt{x}}(h_{i-2^{k-1},k-1}(\sqrt{x})(1+\sqrt{x})^{mk}-h_{i-2^{k-1},k-1}(-\sqrt{x})(1-\sqrt{x})^{mk}), &  & \mbox{for}\quad i=2^{k-1},\ldots,2^{k}-1
\end{array}
\end{cases}.
\end{equation*}
\end{lem}
\begin{proof}
We proceed by induction on $k$. We have the obvious equality $h_{0,0}(x)=1$. Let us suppose that our result is true for $k$. We then have for $i=0,1,\ldots,2^{k-1}-1$:
\begin{align*}
H_{i,k+1}(x^2) &=\sum_{n=0}^{\infty}b_{m}(2^{k+1}n+i)x^{2n}\\
               &=\sum_{n=0}^{\infty}b_{m}(2^{k}\cdot 2n+i)x^{2n}\\
               &=\frac{1}{2}(H_{i,k}(x)+H_{i,k}(-x))\\
               &=\frac{1}{2}\left(\frac{h_{i,k}(x)}{(1-x)^{mk}}H_{m}(x)+\frac{h_{i,k}(-x)}{(1+x)^{mk}}H_{m}(-x)\right)\\
               &=\frac{1}{2}\left(\frac{h_{i,k}(x)}{(1-x)^{mk}}\cdot\frac{1}{(1-x)^{m}}+\frac{h_{i,k}(-x)}{(1+x)^{mk}}\cdot\frac{1}{(1+x)^{m}}\right)H_{m}(x^2)\\
               &=\frac{1}{2}\left((1+x)^{m(k+1)}h_{i,k}(x)+(1-x)^{m(k+1)}h_{i,k}(-x)\right)\frac{H_{m}(x^2)}{(1-x^2)^{m(k+1)}}\\
               &=\frac{h_{i,k+1}(x^2)}{(1-x^2)^{m(k+1)}}H_{m}(x^2).
\end{align*}

It is clear that $\frac{1}{2}\left((1+x)^{m(k+1)}h_{i,k}(x)+(1-x)^{m(k+1)}h_{i,k}(-x)\right)$ is an even function of $x$. Comparing now the first and the last expression from the above and replacing $x$ by $x^{1/2}$, we get the statement of our lemma. If $i=2^{k-1},\ldots,2^{k}-1$ then we write $i=2^k+i'$, where $i'=0,1,\ldots,2^{k-1}-1$, and perform analogous calculations.
\begin{align*}
H_{i,k+1}(x^2) &=\sum_{n=0}^{\infty}b_{m}(2^{k+1}n+i)x^{2n}\\
               &=\sum_{n=0}^{\infty}b_{m}(2^{k}\cdot 2n+2^k+i')x^{2n}\\
               &=\frac{1}{x}\sum_{n=0}^{\infty}b_{m}(2^{k}\cdot (2n+1)+i')x^{2n+1}\\
               &=\frac{1}{2x}(H_{i',k}(x)-H_{i',k}(-x))\\
               &=\frac{1}{2x}\left(\frac{h_{i',k}(x)}{(1-x)^{mk}}H_{m}(x)-\frac{h_{i',k}(-x)}{(1+x)^{mk}}H_{m}(-x)\right)\\
               &=\frac{1}{2x}\left(\frac{h_{i',k}(x)}{(1-x)^{mk}}\cdot\frac{1}{(1-x)^{m}}-\frac{h_{i',k}(-x)}{(1+x)^{mk}}\cdot\frac{1}{(1+x)^{m}}\right)H_{m}(x^2)\\
               &=\frac{1}{2x}\left((1+x)^{m(k+1)}h_{i',k}(x)-(1-x)^{m(k+1)}h_{i',k}(-x)\right)\frac{H_{m}(x^2)}{(1-x^2)^{m(k+1)}}\\
               &=\frac{h_{i,k+1}(x^2)}{(1-x^2)^{m(k+1)}}H_{m}(x^2).
\end{align*}
Replacing $x$ by $x^{\frac{1}{2}}$, we get the result.
\end{proof}

The above lemma is a useful tool which sometimes allows to get congruences involving the (sub)sequence $(b_{m}(2^{k}n+i))_{n\in\N}$. We have the following results concerning the behaviour of $F_{i,k}(x)\pmod{p}$, where $p$ is a prime number, $m=p^{s}$ for some $s\in\N_{+}$ and $k=1,2,3$.

\begin{lem}
Let $p$ be an odd prime number and $m=p^{s}$ for some $s\in\N_{+}$.

For $k=1$ the following congruences hold:
\begin{align*}
&h_{0,1}(x)\equiv 1\pmod{p},\\
&h_{1,1}(x)\equiv x^{\frac{m-1}{2}}\pmod{p}.
\end{align*}

For $k=2$ the following congruences hold:
\begin{align*}
&h_{0,2}(x)\equiv 1+x^{m}\pmod{p},\\
&h_{1,2}(x)\equiv
\begin{cases}
\begin{array}{lll}
x^{\frac{m-1}{2}}+x^{\frac{5m-1}{4}}\pmod{p}, &  &\mbox{if}\; m\equiv 1\pmod{4}  \\
2x^{\frac{3m-1}{4}}\pmod{p}, &  &\mbox{if}\; m\equiv 3\pmod{4}
\end{array}
\end{cases},\\
&h_{2,2}(x)\equiv 2x^{\frac{m-1}{2}}\pmod{p},\\
&h_{3,2}(x)\equiv
\begin{cases}
\begin{array}{lll}
2x^{\frac{3}{4}(m-1)}\pmod{p}, &  & \mbox{if}\; m\equiv 1\pmod{4}\\
x^{\frac{m-3}{4}}+x^{\frac{5m-3}{4}}\pmod{p}, &  &\mbox{if}\; m\equiv 3\pmod{4}
\end{array}
\end{cases}.
\end{align*}
\end{lem}
\begin{proof}
Direct calculations give:
\begin{align*}
&h_{0,1}(x)=\frac{1}{2}\left((1+\sqrt{x})^m+(1-\sqrt{x})^m\right)\equiv \frac{1}{2}\left(1+x^{\frac{m}{2}}+1-x^{\frac{m}{2}}\right)=1\pmod{p},\\
&h_{1,1}(x)=\frac{1}{2\sqrt{x}}\left((1+\sqrt{x})^m-(1-\sqrt{x})^m\right)\equiv \frac{1}{2\sqrt{x}}\left(1+x^{\frac{m}{2}}-1+x^{\frac{m}{2}}\right)=x^{\frac{m-1}{2}}\pmod{p},
\end{align*}
\begin{equation*}
\begin{split}
h_{0,2}(x) &=\frac{1}{2}\left(h_{0,1}(\sqrt{x})(1+\sqrt{x})^{2m}+h_{0,1}(-\sqrt{x})(1-\sqrt{x})^{2m}\right)\equiv\frac{1}{2}\left((1+\sqrt{x})^{2m}+(1-\sqrt{x})^{2m}\right) \\
&\equiv\frac{1}{2}\left(\left(1+x^{\frac{m}{2}}\right)^2+\left(1-x^{\frac{m}{2}}\right)^2\right)=\frac{1}{2}\left(1+2x^{\frac{m}{2}}+x^m+1-2x^{\frac{m}{2}}+x^m\right)=1+x^m\pmod{p},
\end{split}
\end{equation*}
\begin{align*}
h_{2,2}(x) &=\frac{1}{2\sqrt{x}}\left(h_{0,1}(\sqrt{x})(1+\sqrt{x})^{2m}-h_{0,1}(-\sqrt{x})(1-\sqrt{x})^{2m}\right)\equiv\frac{1}{2\sqrt{x}}\left((1+\sqrt{x})^{2m}-(1-\sqrt{x})^{2m}\right) \\
&\equiv\frac{1}{2\sqrt{x}}\left(\left(1+x^{\frac{m}{2}}\right)^2-\left(1-x^{\frac{m}{2}}\right)^2\right)=\frac{1}{2\sqrt{x}}\left(1+2x^{\frac{m}{2}}+x^m-1+2x^{\frac{m}{2}}-x^m\right)=2x^{\frac{m-1}{2}}\pmod{p},
\end{align*}
\begin{equation*}
\begin{split}
h_{1,2}(x) &=\frac{1}{2}\left(h_{1,1}(\sqrt{x})(1+\sqrt{x})^{2m}+h_{1,1}(-\sqrt{x})(1-\sqrt{x})^{2m}\right) \\
&\equiv\frac{1}{2}\left(x^{\frac{m-1}{4}}(1+\sqrt{x})^{2m}+(-x)^{\frac{m-1}{4}}(1-\sqrt{x})^{2m}\right) \\
&\equiv\frac{1}{2}\left(x^{\frac{m-1}{4}}\left(1+x^{\frac{m}{2}}\right)^2+(-x)^{\frac{m-1}{4}}\left(1-x^{\frac{m}{2}}\right)^2\right) \\
&=\frac{1}{2}\left(x^{\frac{m-1}{4}}\left(1+2x^{\frac{m}{2}}+x^m\right)+(-x)^{\frac{m-1}{4}}\left(1-2x^{\frac{m}{2}}+x^m\right)\right) \\
&=
\begin{cases}
x^{\frac{m-1}{4}}+x^{\frac{5m-1}{4}} \pmod{p}, & \mbox{ if } m\equiv 1\pmod{4} \\
2x^{\frac{3m-1}{4}} \pmod{p}, & \mbox{ if } m\equiv 3\pmod{4}
\end{cases},
\end{split}
\end{equation*}
\begin{equation*}
\begin{split}
h_{3,2}(x) &=\frac{1}{2\sqrt{x}}\left(h_{1,1}(\sqrt{x})(1+\sqrt{x})^{2m}-h_{1,1}(-\sqrt{x})(1-\sqrt{x})^{2m}\right) \\
&\equiv\frac{1}{2\sqrt{x}}\left(x^{\frac{m-1}{4}}(1+\sqrt{x})^{2m}-(-x)^{\frac{m-1}{4}}(1-\sqrt{x})^{2m}\right) \\
&\equiv\frac{1}{2\sqrt{x}}\left(x^{\frac{m-1}{4}}\left(1+x^{\frac{m}{2}}\right)^2-(-x)^{\frac{m-1}{4}}\left(1-x^{\frac{m}{2}}\right)^2\right) \\
&=\frac{1}{2\sqrt{x}}\left(x^{\frac{m-1}{4}}\left(1+2x^{\frac{m}{2}}+x^m\right)-(-x)^{\frac{m-1}{4}}\left(1-2x^{\frac{m}{2}}+x^m\right)\right) \\
&=
\begin{cases}
2x^{\frac{3m-3}{4}} \pmod{p}, & \mbox{ if } m\equiv 1\pmod{4} \\
x^{\frac{m-3}{4}}+x^{\frac{5m-3}{4}} \pmod{p}, & \mbox{ if } m\equiv 3\pmod{4}
\end{cases}.
\end{split}
\end{equation*}
\end{proof}

As a consequence of the above lemma we get the following interesting

\begin{cor}\label{congrup}
Let $p$ be an odd prime number and $m=p^{s}$ for some $s\in\N_{+}$. Then the following congruences are true:
\begin{equation*}
b_{m}(2n+i)-b_{m}(2(n-m)+i)\equiv
\begin{cases}
\begin{array}{lll}
b_{m}(n)\pmod{p} &  & \mbox{for}\quad i=0 \\
b_{m}\left(n-\frac{m-1}{2}\right)\pmod{p} &  & \mbox{for}\quad i=1
\end{array}
\end{cases},
\end{equation*}
\begin{align*}
b_{m}(4n+i)-&2b_{m}(4(n-m)+i)+b_{m}(4(n-2m)+i)\equiv \\
&\begin{cases}
\begin{array}{lll}
  b_{m}(n)+b_{m}(n-m) &  & \mbox{for}\quad i=0 \\
  b_{m}\left(n-\frac{m-1}{4}\right)+b_{m}\left(n-\frac{5m-1}{4}\right)&  & \mbox{for}\quad i=1\quad\mbox{and}\quad s\equiv 0\pmod{2} \\
  2b_{m}\left(n-\frac{3m-1}{4}\right) &  & \mbox{for}\quad i=1\quad\mbox{and}\quad s\equiv 1\pmod{2} \\
  2b_{m}\left(n-\frac{m-1}{2}\right)  &  & \mbox{for}\quad i=2 \\
  2b_{m}\left(n-\frac{3}{4}(m-1)\right)) &  & \mbox{for}\quad i=3\quad\mbox{and}\quad s\equiv 0\pmod{2} \\
  b_{m}\left(n-\frac{m-3}{4}\right)+b_{m}\left(n-\frac{5m-3}{4}\right) &  & \mbox{for}\quad i=3\quad\mbox{and}\quad s\equiv 0\pmod{2}
\end{array}
\end{cases}\pmod{p}.
\end{align*}
\end{cor}

\begin{proof}
In order to obtain the first congruence it suffices to compare the coefficients of functions $(1-x^m)H_{i,1}(x)\equiv (1-x)^mH_{i,k}(x)\pmod{m}$ and $h_{i,1}(x)H_m(x)$, which are equivalent modulo $m$ by the previous lemma. For the second one we compare the functions $(1-2x^m+x^{2m})H_{i,2}(x)\equiv (1-x)^{2m}H_{i,k}(x)\pmod{m}$ and $h_{i,2}(x)H_m(x)$.
\end{proof}

Using a different approach we get the following:

\begin{thm}\label{4.10}
Let $m\in\N_{\geq 2}$. Then for $n\in\N$ the following identity holds:
\begin{equation*}
nb_{m}(n)=m\sum_{i=0}^{n}(n-i)b_{1}(n-i)b_{m-1}(i).
\end{equation*}
In particular, if $\gcd(m,n)=1$ then
$$
b_{m}(n)\equiv 0\pmod{m}.
$$
\end{thm}
\begin{proof}
We have the following equality: $xH_{m}'(x)=x(H(x)^{m})'=mxH'(x)(H(x))^{m-1}=mxH'(x)H_{m-1}(x)$. Equivalently:
$$
\sum_{n=0}^{\infty}nb_{m}(n)x^{n}=m\left(\sum_{n=0}^{\infty}nb_{i}(n)x^{n}\right)\left(\sum_{n=0}^{\infty}b_{m-1}(n)\right)=\sum_{n=0}^{\infty}\left(m\sum_{i=0}^{n}(n-i)b_{1}(n-i)b_{m-1}(i)\right)x^{n}.
$$
Comparing now the coefficients on the both sides of the above identity, we get the first statement of our theorem. The second one is immediate.
\end{proof}

\begin{rem}
{\rm Let $p$ be an odd prime number and $m=p^{s}$ for some $s\in\N_{+}$. Let us observe that for each $n\in\N$ the following congruence holds:
\begin{equation}\label{cong:p}
b_{m}((2n+1)m)\equiv b_{m}(2nm)\pmod{p}.
\end{equation}
Indeed, we have
\begin{equation*}
H_{p^{s}}(x)\equiv H_{1}(x)^{p^{s}}\equiv H_{1}(x^{p^{s}})\pmod{p}
\end{equation*}
and thus
\begin{equation*}
b_{p^{s}}(n)\equiv
\begin{cases}
\begin{array}{lll}
 0 & \pmod{p} & \mbox{if}\quad n\not\equiv 0\pmod{p^{s}} \\
 b_{1}\left(\frac{n}{p^{s}}\right) & \pmod{p} & \mbox{if}\quad n\equiv 0\pmod{p^{s}}
\end{array}
\end{cases}.
\end{equation*}
From the above congruence, we have $b_{m}(nm)\equiv b_1(n)\pmod{p}$ for $n\in\N$. Since $b_{1}(2n+1)=b_{1}(2n)$, we obtain (\ref{cong:p}). Actually, in the same way we can prove that for $r,s\in\N_+$, a prime number $p$ and $n\in\N$ we have
\begin{equation*}
b_{rp^s}(n)\equiv
\begin{cases}
0, & \mbox{ if } n\not\equiv 0\pmod{p^s}\\
b_r(\frac{n}{p^s})\pmod{p}, & \mbox{ if } n\equiv 0\pmod{p^s}
\end{cases}.
\end{equation*}
In particular, if $m=p_1^{s_1}\cdots p_k^{s_k}$ is the factorization of a given positive integer $m$, $n\in\N$ and $p_i^{s_i}\nmid n$ for each $i\in\{1,...,k\}$ then $p_1\cdots p_k\mid b_m(n)$.}
\end{rem}

\begin{rem}
{\rm The method used in the proof of Corollary \ref{congrup} can be also used in order to get some congruences involving sums of certain values of $b_{m}$ modulo primes $p$ which are co-prime to $m$. Indeed, the definition of the function $H_{i,k,m}$ and the corresponding polynomials $h_{i,k,m}$ guarantees the identity
$$
(1-x)^{km}\sum_{n=0}^{\infty}b_{m}(2^{k}n+i)=h_{i,k,m}(x)H_{m}(x)
$$
and the congruence
$$
(1-x)^{km}\sum_{n=0}^{\infty}b_{m}(2^{k}n+i)\equiv h_{i,k,m}(x)H_{m}(x)\pmod{p},
$$
where $p$ is a given prime number. If we are lucky, the reduction $h_{i,k,m}(x)\pmod{p}$ is simple, i.e., contains far fewer non-zero coefficients $\pmod{p}$ than the non-reduced polynomial, and we can deduce new congruences with $i,k,m$ and $p$ satisfying $\gcd(p,m)=1$. Using this approach, one can find many congruences with few summands and, in fact, there are some instances of $i, k, m$ and $p$ such that $h_{i,k,m}(x)\equiv cx^{s}\pmod{p}$ for some $c\in\Z$ and $s\in\N$. In particular we get the following
\begin{thm}\label{4.13}
The following congruences are true:
\begin{equation*}
\begin{array}{lll}
  \sum_{i=0}^{8}b_{4}(4(n-i)+1) & \equiv & b_{4}(n)\pmod{3}, n\geq 8, \\
  \sum_{i=0}^{4}b_{2}(4(n-i)) & \equiv & b_{2}(n)\pmod{5}, n\geq 4, \\
  \sum_{i=0}^{4}b_{2}(4(n-i)+2) & \equiv & b_{2}(n-2)\pmod{5}, n\geq 4.
\end{array}
\end{equation*}
\end{thm}
\begin{proof}
The congruences given above are consequences of the following equalities
\begin{align*}
h_{1,2,4}(x)&=4(3x+1)(3x^3+27x^2+33x+1)\equiv 1\pmod{3},\\
h_{0,2,2}(x)&=5 x^2+10 x+1\equiv 1\pmod{5},\\
h_{2,2,2}(x)&=x^2+10 x+5\equiv x^2\pmod{5}
\end{align*}
and the fact that $(1-x)^{8}\equiv \frac{1-x^9}{1-x}\pmod{3}$ and $(1-x)^{4}\equiv \frac{1-x^{5}}{1-x}\pmod{5}$.
\end{proof}
}
\end{rem}

It is quite interesting to look at the family of polynomials $(h_{i,k,m}(x))$ as an independent object of study and to ask which members of this family are reducible. It seems that this is a rather difficult question. Based on numerical observations, we state the following

\begin{thm}\label{8(x+1)}
We have
\begin{align*}
&h_{2^{k}+1,k+1,2}(x)\equiv 0\pmod{8(x+1)},\quad k\in\N_{+},\\
&h_{2^{k}+2,k+1,4}(x)\equiv 0\pmod{8(x+1)},\quad k\in\N_{\geq 2}.
\end{align*}
\end{thm}

In order to prove the above theorem, we will show three lemmas.

\begin{lem}\label{4div}
For every $n,j\in\N$ we have $4\mid {4n\choose 2j}-{2n\choose j}$.
\end{lem}

\begin{proof}
We write
\begin{align*}
{4n\choose 2j}-{2n\choose j} & ={2n\choose j}\left(\frac{(4n-1)(4n-3)\cdots (4n-2j+1)}{1\cdot 3\cdots (2j-1)}-1\right)\\
& ={2n\choose j}\frac{(4n-1)(4n-3)\cdots (4n-2j+1)-1\cdot 3\cdots (2j-1)}{1\cdot 3\cdots (2j-1)}.
\end{align*}
The factor $(4n-1)(4n-3)\cdots (4n-2j+1)-1\cdot 3\cdots (2j-1)$ is even, since numbers $(4n-1)(4n-3)\cdots (4n-2j+1)$ and $1\cdot 3\cdots (2j-1)$ are odd. If $j$ is even then the products $(4n-1)(4n-3)\cdots (4n-2j+1)$ and $1\cdot 3\cdots (2j-1)$ have the same number of factors congruent to $3$ modulo $4$, hence they are congruent modulo $4$ and their difference is divisible by $4$. If $j$ is odd then
\begin{equation*}
\nu_2\left(2n\choose j\right)=\nu_2\left(\frac{2n}{j}{2n-1\choose j-1}\right)\geq 1.
\end{equation*}
Thus the numbers $2n\choose j$ and $(4n-1)(4n-3)\cdots (4n-2j+1)-1\cdot 3\cdots (2j-1)$ are even which implies that $4\mid {4n\choose 2j}-{2n\choose j}$.
\end{proof}

The lemma given below is well known and we omit of its proof.

\begin{lem}\label{palind}
We will call a polynomial $P(x)=\sum_{j=s}^d c_jx^j\in\Z[x]$, $c_s,c_d\neq 0$, palindromic, if $c_{d-j}=c_{s+j}$ for each $j\in\{0,1,...,d-s\}$. Each palindromic polynomial $P\in\Z[x]$ of order $s$ and degree $d$ can be uniquely written in the form (we recall that if $P(x)=\sum_{j=s}^d a_jx^j$, where $a_s\neq 0$, then we define the order of polynomial $P$ as the number $\ord P=s$)
\begin{equation*}
P(x)=\sum_{j=0}^{\left\lfloor\frac{d-s}{2}\right\rfloor} a_{s+j}x^{s+j}(1+x)^{d-s-2j},
\end{equation*}
where $a_{s+j}\in\Z$. Then $a_s=c_s$.
Moreover, if polynomials $P_1, P_2, ..., P_r$ are palindromic and there exists $c\in\N$ such that $\ord P_i+\deg P_i=c$ for every $i\in\{1,2,...,r\}$ then the polynomial $\sum_{i=1}^r P_i$ is palindromic.
\end{lem}

\begin{lem}\label{h12}
We have
\begin{align*}
h_{1,k+1,2}(x) & =\sum_{j=0}^k a_{j,k}x^j(1+x)^{2k-2j},\quad k\in\N,\\
h_{2,k+1,4}(x) & =\sum_{j=0}^k b_{j,k}x^j(1+x)^{2k-2j},\quad k\in\N_+,
\end{align*}
where $a_{0,k}=2$, $b_{0,k}=14$, $8\mid a_{j,k}$ and $8\mid b_{j,k}$ for $j>0$.
\end{lem}

\begin{proof}
We proceed by induction on $k$. If $k=0$ then $h_{1,0+1,2}(x)=2$. Assume now that $h_{1,k+1,2}(x)=\sum_{j=0}^k a_{j,k}x^j(1+x)^{2k-2j}$ for some $k\in\N$, where $a_{0,k}=2$ and $8\mid a_{j,k}$ for $j>0$. Then
\begin{equation}\label{hrec}
\begin{split}
h_{1,k+2,2}(x)&=\frac{1}{2}\left((1+\sqrt{x})^{2k+4}h_{1,k+1,2}(\sqrt{x})+(1-\sqrt{x})^{2k+4}h_{1,k+1,2}(-\sqrt{x})\right)\\
&=\sum_{j=0}^k \frac{1}{2}a_{j,k}\left(\sqrt{x}^j(1+\sqrt{x})^{4k-2j+4}+(-\sqrt{x})^j(1-\sqrt{x})^{4k-2j+4}\right).
\end{split}
\end{equation}
For each $j\in\{0,1,...,k\}$ the expression $\frac{1}{2}a_{j,k}\left(\sqrt{x}^j(1+\sqrt{x})^{4k-2j+4}+(-\sqrt{x})^j(1-\sqrt{x})^{4k-2j+4}\right)$ is a palindromic polynomial in $\Z[x]$ of degree $2k-j+\left\lfloor\frac{j}{2}\right\rfloor+2$ and order $\left\lceil\frac{j}{2}\right\rfloor$. Hence the sum of degree and  the order of  the $j$th summand is equal to $2k+2$ and by Lemma \ref{palind} the polynomial $h_{1,k+2,2}(x)$ is palindromic. The degree of $h_{1,k+2,2}$ is $2k+2$ and its leading coefficient is equal to $a_{0,k}=2$, thus by Lemma \ref{palind} this polynomial can be uniquely written in the form $h_{1,k+2,2}(x)=\sum_{j=0}^{k+1} a_{j,k+1}x^j(1+x)^{2k+2-2j}$, where $a_{0,k+1}=2$. It suffices to show that $8\mid a_{j,k+1}$ for $j>0$. Since $8\mid a_{j,k}$ for $j>0$, the $j$th summand in (\ref{hrec}), $j>0$, has no influence on the values $a_{j,k+1}\pmod{8}$, so we can skip them. Let us consider the $0$th summand in (\ref{hrec}).
\begin{align*}
&\left((1+\sqrt{x})^{4k+4}+(1-\sqrt{x})^{4k+4}\right)=2\sum_{j=0}^{2k+2} {4k+4\choose 2j}x^j\\
=&2(1+x)^{2k+2}+2\sum_{j=1}^{2k+1} \left({4k+4\choose 2j}-{2k+2\choose j}\right)x^j=2(1+x)^{2k+2}+8x\sum_{j=0}^{2k} \alpha_jx^j\\
=& 2(1+x)^{2k+2}+8x\sum_{j=0}^{k} a'_jx^j(1+x)^{2k-j},
\end{align*}
where we use Lemma \ref{4div} to write ${4k+4\choose 2j}-{2k+2\choose j}=4\alpha_j$ and clearly $\alpha_{2k-j}=\alpha_j$. Finally, $a_{j,k+1}\equiv 8a'_{j-1,k+1}\equiv 0\pmod{8}$.

The proof for polynomials $h_{2,k+1,4}$ is analogous, so we leave its details for the reader.
\end{proof}

Now we are ready to prove Theorem \ref{8(x+1)}.

\begin{proof}[Proof of Theorem \ref{8(x+1)}]
Since the proof for polynomials $h_{2^k+2,k+1,4}$ is completely analogous, we will only present the proof for polynomials $h_{2^k+1,k+1,2}$.

We compute $h_{2^k+1,k+1,2}$ using Lemma \ref{h12}.
\begin{equation*}
\begin{split}
h_{2^k+1,k+1,2}&=\frac{1}{2\sqrt{x}}\left((1+\sqrt{x})^{2k+2}h_{1,k,2}(\sqrt{x})-(1-\sqrt{x})^{2k+2}h_{1,k,2}(-\sqrt{x})\right)\\
&=\sum_{j=0}^{k-1} \frac{1}{2\sqrt{x}}a_{j,k-1}\left(\sqrt{x}^j(1+\sqrt{x})^{4k-2j}-(-\sqrt{x})^j(1+\sqrt{x})^{4k-2j}\right).
\end{split}
\end{equation*}
It remains to show that $4(x+1)\mid \frac{1}{2\sqrt{x}}\left(\sqrt{x}^j(1+\sqrt{x})^{4k-2j}-(-\sqrt{x})^j(1+\sqrt{x})^{4k-2j}\right)$ for even $j$ and $x+1\mid \frac{1}{2\sqrt{x}}\left(\sqrt{x}^j(1+\sqrt{x})^{4k-2j}-(-\sqrt{x})^j(1+\sqrt{x})^{4k-2j}\right)$ for odd $j$, since $2\mid a_{0,k-1}$ and $8\mid a_{j,k-1}$ for $j>0$. If $j$ is even then
$$\frac{1}{2\sqrt{x}}\left(\sqrt{x}^j(1+\sqrt{x})^{4k-2j}-(-\sqrt{x})^j(1+\sqrt{x})^{4k-2j}\right)=\frac{1}{2\sqrt{x}}x^{\frac{j}{2}}\left((1+\sqrt{x})^{4k-2j}-(1+\sqrt{x})^{4k-2j}\right)$$
is divisible by $\frac{1}{2\sqrt{x}}\left((1+\sqrt{x})^4-(1-\sqrt{x})^4\right)=4(1+x)$. If $j$ is odd then
$$\frac{1}{2\sqrt{x}}\left(\sqrt{x}^j(1+\sqrt{x})^{4k-2j}-(-\sqrt{x})^j(1+\sqrt{x})^{4k-2j}\right)=\frac{1}{2}x^{\frac{j-1}{2}}\left((1+\sqrt{x})^{4k-2j}-(1+\sqrt{x})^{4k-2j}\right)$$
is divisible by $\frac{1}{2}\left((1+\sqrt{x})^2+(1-\sqrt{x})^2\right)=1+x$.
\end{proof}

Let $i\in\{0,1\}$ and observe that the generating function for the sequence $(h_{i,1,m}(x))_{m\in\N}$ is rational. More precisely, we have
\begin{align*}
&G_{0,1}(x,T)=\sum_{m=0}^{\infty}h_{0,1,m}(x)T^{m}=\sum_{m=0}^{\infty}\frac{1}{2}((1+\sqrt{x})^{m}+(1-\sqrt{x})^{m})T^{m}=\frac{T-1}{(x-1)T^2+2T-1},\\
&G_{1,1}(x,T)=\sum_{m=0}^{\infty}h_{1,1,m}(x)T^{m}=\sum_{m=0}^{\infty}\frac{1}{2\sqrt{x}}((1+\sqrt{x})^{m}-(1-\sqrt{x})^{m})T^{m}=-\frac{T}{(x-1)T^2+2T-1}.
\end{align*}
This suggests that for any given $k\in\N_{+}$ and $i\in\{0,\ldots, 2^{k}-1\}$ the sequence $(h_{i,k,m}(x))_{m\in\N}$ should satisfy a linear recurrence. We confirm this in the following result.

\begin{thm}\label{recurrence}
Let $k\in\N_{+}$ and $i\in\{0,\ldots,2^{k}-1\}$ and define
$$
G_{i,k}(x,T)=\sum_{m=0}^{\infty}h_{i,k,m}(x)T^{m}.
$$
Then the function $G_{i,k}$ is rational (as a function in two variables $x, T$). In particular the sequence $(h_{i,k,m}(x))_{m\in\N}$ is annihilated by the difference operator $V_{k}$ defined recursively in the following way:
\begin{equation*}
V_{1}(x,\theta)=(x-1)\theta ^2+2\theta-1,\quad V_{k+1}(x,\theta)=V_{k}(\sqrt{x},(1+\sqrt{x})^{k+1}\theta)V_{k}(-\sqrt{x},(1-\sqrt{x})^{k+1}\theta),
\end{equation*}
where $\theta((a_{n})_{n\in\N_{\geq r}})=(a_{n-1})_{n\in\N_{\geq r+1}}$.
\end{thm}
\begin{proof}
If $k=1$ then the function $G_{i,1}(x,T)$ is rational with respect to variables $x$ and $T$ for $i\in\{0,1\}$. Assume now that $G_{i,k}(x,T)=\frac{P_{i,k}(x,T)}{Q_{i,k}(x,T)}$ for some $k\in\N_+$ and every $i\in\{0,1,...,2^k-1\}$, where $P_{i,k},Q_{i,k}\in\Z[x,T]$, $Q_{i,k}\neq 0$. Then we use the recurrence from Lemma \ref{hik}. We thus have for $i\in\{0,1,...,2^k-1\}$:
\begin{align*}
G_{i,k+1}(x,T) & =\sum_{m=0}^{\infty}h_{i,k+1,m}(x)T^{m}\\
& =\sum_{m=0}^{\infty}\frac{1}{2}\left(h_{i,k,m}(\sqrt{x})(1+\sqrt{x})^{m(k+1)}+h_{i,k,m}(-\sqrt{x})(1-\sqrt{x})^{m(k+1)}\right)T^{m}\\
& =\frac{1}{2}\left(G_{i,k}(\sqrt{x},(1+\sqrt{x})^{k+1}T)+G_{i,k}(-\sqrt{x},(1-\sqrt{x})^{k+1}T)\right)\\
& =\frac{1}{2}\left(\frac{P_{i,k}(\sqrt{x},(1+\sqrt{x})^{k+1}T)}{Q_{i,k}(\sqrt{x},(1+\sqrt{x})^{k+1}T)}+\frac{P_{i,k}(-\sqrt{x},(1-\sqrt{x})^{k+1}T)}{Q_{i,k}(-\sqrt{x},(1-\sqrt{x})^{k+1}T)}\right)\\
& =\frac{\hat{P}_{i,k+1}(\sqrt{x},T)}{\hat{Q}_{i,k+1}(\sqrt{x},T)},\\
&\\
G_{i+2^k,k+1}(x,T) & =\sum_{m=0}^{\infty}h_{i+2^k,k+1,m}(x)T^{m}\\
& =\sum_{m=0}^{\infty}\frac{1}{2\sqrt{x}}\left(h_{i,k,m}(\sqrt{x})(1+\sqrt{x})^{m(k+1)}-h_{i,k,m}(-\sqrt{x})(1-\sqrt{x})^{m(k+1)}\right)T^{m}\\
& =\frac{1}{2\sqrt{x}}\left(G_{i,k}(\sqrt{x},(1+\sqrt{x})^{k+1}T)-G_{i,k}(-\sqrt{x},(1-\sqrt{x})^{k+1}T)\right)\\
& =\frac{1}{2\sqrt{x}}\left(\frac{P_{i,k}(\sqrt{x},(1+\sqrt{x})^{k+1}T)}{Q_{i,k}(\sqrt{x},(1+\sqrt{x})^{k+1}T)}-\frac{P_{i,k}(-\sqrt{x},(1-\sqrt{x})^{k+1}T)}{Q_{i,k}(-\sqrt{x},(1-\sqrt{x})^{k+1}T)}\right)\\
& =\frac{\hat{P}_{i+2^k,k+1}(\sqrt{x},T)}{\hat{Q}_{i+2^k,k+1}(\sqrt{x},T)},
\end{align*}
where
\footnotesize
\begin{equation}\label{polyrec}
\begin{split}
\hat{P}_{i,k+1}(\sqrt{x},T) & =\frac{1}{2}\left(P_{i,k}(\sqrt{x},(1+\sqrt{x})^{k+1}T)Q_{i,k}(-\sqrt{x},(1-\sqrt{x})^{k+1}T)+P_{i,k}(-\sqrt{x},(1-\sqrt{x})^{k+1}T)Q_{i,k}(\sqrt{x},(1+\sqrt{x})^{k+1}T)\right),\\
\hat{P}_{i+2^k,k+1}(\sqrt{x},T) & =\frac{1}{2\sqrt{x}}P_{i,k}(\sqrt{x},(1+\sqrt{x})^{k+1}T)Q_{i,k}(-\sqrt{x},(1-\sqrt{x})^{k+1}T)-P_{i,k}(-\sqrt{x},(1-\sqrt{x})^{k+1}T)Q_{i,k}(\sqrt{x},(1+\sqrt{x})^{k+1}T),\\
\hat{Q}_{i,k+1}(\sqrt{x},T) & =\hat{Q}_{i+2^k,k+1}(\sqrt{x},T)=Q_{i,k}(\sqrt{x},(1+\sqrt{x})^{k+1}T)Q_{i,k}(-\sqrt{x},(1-\sqrt{x})^{k+1}T).
\end{split}
\end{equation}
\normalsize
One can easily check that the polynomials $\hat{P}_{i,k+1}$, $\hat{Q}_{i,k+1}$, $\hat{P}_{i+2^k,k+1}$, $\hat{Q}_{i+2^k,k+1}\in\Z[\sqrt{x},T]$, treated as functions, are even with respect to the first variable. Hence
\begin{equation}\label{polyrec2}
\begin{split}
\hat{P}_{i,k+1}(\sqrt{x},T) & =P_{i,k+1}(x,T),\\
\hat{P}_{i+2^k,k+1}(\sqrt{x},T) & =P_{i+2^k,k+1}(x,T),\\
\hat{Q}_{i,k+1}(\sqrt{x},T) & =Q_{i,k+1}(x,T),\\
\hat{Q}_{i+2^k,k+1}(\sqrt{x},T) & =Q_{i+2^k,k+1}(x,T)
\end{split}
\end{equation}
and as a result $G_{i,k+1}(x,T)=\frac{P_{i,k+1}(x,T)}{Q_{i,k+1}(x,T)}$ and $G_{i+2^k,k+1}(x,T)=\frac{P_{i+2^k,k+1}(x,T)}{Q_{i,k+1}(x,T)}$. Since $\deg P_{i,0}<\deg Q_{i,0}$ for $i\in\{0,1\}$, by formulae (\ref{polyrec}) and (\ref{polyrec2}) we conclude that $\deg P_{i,k}<\deg Q_{i,k}$ for any $k\in\N$ and $i\in\{0,1,...,2^k-1\}$. Hence the equality $Q_{i,k}(x,T)G_{i,k}(x,T)=P_{i,k}(x,T)$ implies annihilation of the sequence $(h_{i,k,m}(x))_{m\in\N}$ by the operator $V_k(x,\theta)=Q_{i,k}(x,\theta)$.
\end{proof}

\bigskip

\section{Questions, remarks and conjectures}

In this section we present several questions and conjectures which appeared during our work on this paper. We also present some related results.

We proved that the sequence $\nu_{2}(t_{2^{k}}(n))$ is 2-regular. This result motivates the following

\begin{ques}
Let $m\in\N_{\geq 2}$ be given. Is the sequence $(\nu_{2}(t_{m}(n))_{n\in\N}$ $2$-regular?
\end{ques}

Numerical computations in Wolfram Mathematica \cite{Wol} make us state the following conjecture on the 2-adic valuation of numbers $t_{2^k+1}(n)$, $n\in\N$, where $k\in\N_2$ is fixed.

\begin{conj}
For each $n\in\N$ we have the following equalities:
\begin{align*}
\nu_2(t_5(4n+j))=4\left\lceil\frac{\nu_2(n+1)}{2}\right\rceil-(\nu_2(n+1)\pmod{2}),\quad j\in\{0,1,2,3\},\\
\nu_2(t_9(8n+j))=5\left\lceil\frac{\nu_2(n+1)}{2}\right\rceil-2(\nu_2(n+1)\pmod{2}),\quad j\in\{0,1,...,7\}.
\end{align*}
In general, for each $k\in\N_2$ there exists a strictly increasing sequence $(A_{k,n})_{n\in\N}$ of nonnegative integers such that $A_{k,0}=0$ and
\begin{equation*}
\nu_2(t_{2^k+1}(2^kn+j))=A_{k,\nu_2(n+1)}
\end{equation*}
for each $n\in\N$ and $j\in\{0,1,...,2^k-1\}$.
\end{conj}

Using Lemma \ref{parityrest} we can verify that $\nu_2(t_{2^k+1}(2^kn+j))=0$ if and only if $n$ is even.

We believe that the following more general statement is true.

\begin{conj}
Let $m\in\N_{\geq 2}$ be given and suppose that $m$ is not of the form $2^{k}-1$ for $k\in\N_{+}$. Then the sequence $(\nu_{2}(b_{m}(n)))_{n\in\N}$ is unbounded.
\end{conj}

We performed extensive calculations in the case of $m=2$ and observed some interesting phenomena concerning the solutions of the equation $\nu_{2}(b_{2}(n))=a$, where $a$ is fixed positive integer. In order to investigate this equation we computed all polynomials $h_{i,k,2}(x)$ for $k\leq 8$ and $i\leq 2^{k}-1$ and looked for these satisfying the congruence $h_{i,k,2}(x)\equiv 0\pmod{2^{a}}$ in $\Z[x]$. For any given polynomial of this type we immediately got the inequality $\nu_{2}(b_{2}(2^{k}n+i))\geq a$ for each $n\in\N$. Then we checked whether the inequality is in essence an equality. In order to verify this it is enough to check whether the following congruence holds
$$
\frac{H_{i,k,2}(x)}{2^{a}}=\frac{h_{i,k,2}(x)}{2^{a}(1-x)^{2k}}H_{2}(x)\equiv \frac{1}{1-x}\pmod{2}.
$$
However, $H_{2}(x)\equiv (1-x)^{2}\pmod{2}$ and thus it is enough to check whether
$$
\frac{h_{i,k,2}(x)}{2^{a}(1-x)^{2(k-1)}}\equiv \frac{1}{1-x}\pmod{2}\Longleftrightarrow \frac{h_{i,k,2}(x)}{2^{a}}\equiv (1-x)^{2k-3}\pmod{2}.
$$
This is easy (at least in the range we considered) and we were able to prove the following
\begin{thm}
The following equalities are true:
\begin{align*}
&\nu_{2}(b_{2}(4n+3))=3,\\
&\nu_{2}(b_{2}(8n+5))=3,\\
&\nu_{2}(b_{2}(16n+i))=3 \;for\; i\in\{6,9,12\},\\
&\nu_{2}(b_{2}(32n+i))=3 \;for\; i\in\{8,17,26\},\\
&\nu_{2}(b_{2}(64n+i))=3 \;for\; i\in\{16,33,50\},\\
&\nu_{2}(b_{2}(128n+i))=3 \;for\; i\in\{32,65,98\},\\
&\nu_{2}(b_{2}(256n+i))=3 \;for\; i\in\{64,129,194\},\\
&\\
\end{align*}
\begin{align*}
&\nu_{2}(b_{2}(32n+i))=4\;for\;i\in\{4,30\},\\
&\nu_{2}(b_{2}(64n+i))=4 \;for\; i\in\{10,56\},\\
&\nu_{2}(b_{2}(128n+i))=4 \;for\; i\in\{48,82\},\\
&\nu_{2}(b_{2}(256n+i))=4 \;for\; i\in\{96,162\},\\
&\\
&\nu_{2}(b_{2}(64n+i))=5 \;for\; i\in\{20,46\},\\
&\nu_{2}(b_{2}(128n+i))=5 \;for\; i\in\{42,88\},\\
&\nu_{2}(b_{2}(256n+i))=5 \;for\; i\in\{18,240\},\\
&\\
&\nu_{2}(b_{2}(128n+i))=6 \;for\; i\in\{14,116\},\\
&\nu_{2}(b_{2}(256n+i))=6 \;for\; i\in\{106,152\},\\
&\\
&\nu_{2}(b_{2}(256n+i))=7 \;for\; i\in\{78,180\}.\\
\end{align*}
\end{thm}

We also expect that the following congruences are true.

\begin{conj}\label{cong:powerof2}
Let $m$ be a fixed positive integer. Then for each $n\in\N$ and $k\geq m+2$ the following congruence holds:
\begin{equation*}
b_{2^{m}}(2^{k+1}n)\equiv b_{2^{m}}(2^{k-1}n)\pmod{2^{k}}.
\end{equation*}
\end{conj}

\begin{conj}\label{cong:powerof2-1}
Let $m$ be a fixed positive integer. Then for each $n\in\N$ and $k\geq m+2$ the following congruence holds:
\begin{equation*}
 b_{2^{m}-1}(2^{k+1}n)\equiv b_{2^{m}-1}(2^{k-1}n)\pmod{2^{4\lfloor \frac{k+1}{2}\rfloor-2}}.
\end{equation*}
\end{conj}

In fact we expect the following.
\begin{conj}\label{cong:general}
 Let $m$ be a fixed positive integer. Then for each $n\in\N$ and given $k\gg 1$ there is a non-decreasing function $f:\N\rightarrow \N$ such that $f(k)=O(k)$ and the following congruence holds
\begin{equation*}
 b_{m}(2^{k+1}n)\equiv b_{m}(2^{k-1}n)\pmod{2^{f(k)}}.
\end{equation*}
\end{conj}


According to numerical computations we noticed that for $m\geq 2$ we have
\begin{equation*}
t_m(3n)>0,\quad t_m(3n+1)<0
\end{equation*}
for most values $n\in\N$. However, $t_m(3n)<0$ for some $n$ and similarly $t_m(3n+1)>0$ for some $n$. Hence we have the following supposition.

\begin{conj}
Let $m$ be a positive integer $\geq 2$. Let us define
\begin{equation*}
\cal{A}_{m,j}=\{n\in\N: {\rm sgn}\; t_m(3n+j)\neq (-1)^j\},\mbox{  } j\in\{0,1\}.
\end{equation*}
Then the sets $\cal{A}_{m,j}$ are infinite and they have asymptotic density equal to $0$, i.e.,
\begin{equation*}
\lim_{n\rightarrow +\infty} \frac{\sharp\cal{A}_{m,j}\cap\{0,1,...,n-1\}}{n}=0.
\end{equation*}
\end{conj}

We also have seen that for initial values $m,n\in\N_+$, $m\geq 2$, the numbers $t_m(n-1)$, $t_m(n)$, $t_m(n+1)$ do not have the same sign, as well.  Moreover we noted that $t_m(n)^2>t_m(n-1)t_m(n+1)$.

\begin{conj}
For any $m,n\in\N_+$, $m\geq 2$, the numbers $t_m(n-1)$, $t_m(n)$, $t_m(n+1)$ do not have the same sign and $t_m(n)^2>t_m(n-1)t_m(n+1)$.
\end{conj}

The above conjecture was proved in Section \ref{Section3} for $m=2$.

We checked that $(-1)^n\left(b_m(n)^2-b_m(n-1)b_m(n+1)\right)>0$ for $m\in\{1,2\}$ and $n\in\N_+$. We noticed that the behaviour of the expressions $b_m(n)^2-b_m(n-1)b_m(n+1)$ for $m>2$ is different. Namely, we expect that following is true.

\begin{conj}
If $m\geq 4$ then $b_m(n)^2-b_m(n-1)b_m(n+1)>0$ for any $n\in\N_+$.

For $m=3$ there exists $n_0\in\N_+$ such that $(-1)^n\left(b_3(n)^2-b_3(n-1)b_3(n+1)\right)>0$ for $n\leq n_0$ and $b_3(n)^2-b_3(n-1)b_3(n+1)>0$ for $n>n_0$.
\end{conj}

\bigskip

{\bf Acknowledgment} The authors are grateful to the referee for a careful reading of the manuscript and valuable suggestions, which improved the
quality of the paper.

\vskip 1cm

\section{Appendix by Andrzej Schinzel}

\begin{lem}\label{t2lem1}
The sequence $(t_{2}(n))_{n\in\N}$ satisfies the following recurrence relation: $t_{2}(0)=1, t_{2}(1)=-2$ and for $n\geq 1$ we have
\begin{equation*}
t_{2}(2n)=t_{2}(n)+t_{2}(n-1),\quad t_{2}(2n+1)=-2t_{2}(n).
\end{equation*}
\end{lem}
\begin{proof}
We have by the formula (\ref{deftm}) with $m=2$
\begin{equation}\label{t2}
t_{2}(n)=\sum_{a+b=n}(-1)^{s_{2}(a)+s_{2}(b)}.
\end{equation}
If $n=2n_{1}+1$, then either $a=2a_{1}+1, b=2b_{1}, a_{1}+b_{1}=n_{1}$ and $s_{2}(a)+s_{2}(b)=s_{2}(a_{1})+s_{2}(b_{1})+1$ or $a=2a_{1}, b=2b_{1}+1, a_{1}+b_{1}=n_{1}$ and then $s_{2}(a)+s_{2}(b)=s_{2}(a_{1})+s_{2}(b_{1})+1$, thus $t_{2}(2n_{1}+1)=-2t_{2}(n_{1})$.

If $n=2n_{1}$, then either $a=2a_{1}, b=2b_{1}, a_{1}+b_{1}=n_{1}$ and $s_{2}(a)+s_{2}(b)=s_{2}(a_{1})+s_{2}(b_{1})$ or $a=2a_{1}+1, b=2b_{1}+1, a_{1}+b_{1}=n_{1}-1$ and $s_{2}(a)+s_{2}(b)=s_{2}(a_{1})+s_{2}(b_{1})$, thus $t_{2}(2n_{1})=t_{2}(n_{1})+t_{2}(n_{1}-1)$.
\end{proof}

\begin{lem}\label{t2lem2}
For $n\in\N$ we have $t_{2}(2n)\equiv 1+2n\pmod{4}$.
\end{lem}
\begin{proof}
By (\ref{t2}) we have
\begin{equation}\label{rel1}
t_{2}(2n)=2\sum_{a+b=2n, a<b}(-1)^{s_{2}(a)+s_{2}(b)}+1
\end{equation}
and clearly
\begin{equation}\label{rel2}
\sum_{a+b=2n, a<b}(-1)^{s_{2}(a)+s_{2}(b)}\equiv |\{(a,b)\in\Z\times\Z:\;0\leq a<b, a+b=2n\}|\equiv n\pmod{2}.
\end{equation}
Lemma \ref{t2lem2} follows from (\ref{rel1}) and (\ref{rel2}).
\end{proof}

We have the following result concerning the existence of solutions of the equation $t_{2}(n)=m$ with fixed $m$.

\begin{thm}\label{Appthm1}
For all integers $n\geq 0$ and $m$ if $t_{2}(n)=m$, then $t_{2}(n)=-t_{2}(n')$, where
\begin{equation*}
n'=n+(-1)^{\nu_{2}(m)+\frac{m-2^{\nu_{2}(m)}}{2^{\nu_{2}(m)+1}}}2^{\nu_{2}(m)+1}.
\end{equation*}
\end{thm}
\begin{proof}
First, we shall show that $n'\geq 0$. Assuming the contrary, we have $n'<0$, thus $n<2^{\nu_{2}(m)+1}$ and since by (\ref{t2}) $n+1\geq|m|$, it follows that $m=\pm 2^{\nu_{2}(m)}, n=2^{\nu_{2}(m)}-1, m=(-2)^{\nu_{2}(m)}$ and finally $n'=2^{\nu_{2}(m)+1}+2^{\nu_{2}(m)}-1>0$, a contradiction.

In order to prove that $t_{2}(n)=-t_{2}(n')$ we proceed by induction on $n$. For $n=0,1,2$ the theorem is true, since $t_{2}(0)=1, t_{2}(1)=-2, t_{2}(2)=-1, t_{2}(5)=2$. Assume now that the theorem is true for all $n<N$ with $N\geq 3$. If $N$ is odd, we have by Lemma \ref{t2lem1}
\begin{equation*}
t_{2}(N)=-2t_{2}\left(\frac{N-1}{2}\right), \quad t_{2}(N')=-2t_{2}\left(\frac{N'-1}{2}\right)
\end{equation*}
and by the induction hypothesis it suffices to show that $\left(\frac{N-1}{2}\right)'=\frac{N'-1}{2}$. However, $t_{2}(\frac{N-1}{2})=-\frac{m}{2}, \nu_{2}\left(-\frac{m}{2}\right)=\nu_{2}(m)-1$ and
\begin{equation*}
(-1)^{\nu_{2}(m)-1+\frac{-\frac{m}{2}-2^{\nu_{2}(m)-1}}{2^{\nu_{2}(m)}}}2^{\nu_{2}(m)}=\frac{1}{2}\left((-1)^{\nu_{2}(m)+\frac{m-2^{\nu_{2}(m)}}{2^{\nu_{2}(m)+1}}}2^{\nu_{2}(m)+1}\right)
\end{equation*}
and our result follows in the case of $N$ odd.

If $N$ is even, $N=2N_{1}, N_{1}\geq 2$ we have by Lemma \ref{t2lem1} the identity $t_{2}(N)=t_{2}(N_{1})+t_{2}(N_{1}-1)$ and by Lemma \ref{t2lem2}
\begin{equation*}
N'=N+(-1)^{\frac{t_{2}(2N_{1})-1}{2}}2=N+(-1)^{N_{1}}2.
\end{equation*}
Therefore, if $N_{1}\equiv 0\pmod{2}$ we have by Lemma \ref{t2lem1}
\begin{equation*}
t_{2}(N')=t_{2}(N+2)=t_{2}(N_{1}+1)+t_{2}(N_{1}),
\end{equation*}
which gives the equivalence
\begin{align*}
t_{2}(N)=-t_{2}(N')&\Longleftrightarrow t_{2}(N_{1}-1)+t_{2}(N_{1}+1)=-2t_{2}(N_{1})\\
                   &\Longleftrightarrow -2t_{2}\left(\frac{N_{1}}{2}-1\right)-2t_{2}\left(\frac{N_{1}}{2}\right)=-2t_{2}(N_{1})\\
                   &\Longleftrightarrow t_{2}(N_{1})=t_{2}\left(\frac{N_{1}}{2}\right)+t_{2}\left(\frac{N_{1}}{2}-1\right).
\end{align*}
The last equality is true by Lemma \ref{t2lem1} with $n=N_{1}/2$.

If $N_{1}\equiv 1\pmod{2}$ we have by Lemma \ref{t2lem1}
\begin{equation*}
t_{2}(N')=t_{2}(N-2)=t_{2}(N_{1}-1)+t_{2}(N_{1}-2),
\end{equation*}
which gives the equivalence
\begin{align*}
t_{2}(N)=-t_{2}(N')&\Longleftrightarrow t_{2}(N_{1})+t_{2}(N_{1}-2)=-2t_{2}(N_{1}-1)\\
                   &\Longleftrightarrow -2t_{2}\left(\frac{N_{1}-1}{2}\right)-2t_{2}\left(\frac{N_{1}-3}{2}\right)=-2t_{2}(N_{1}-1)\\
                   &\Longleftrightarrow t_{2}(N_{1}-1)=t_{2}\left(\frac{N_{1}-1}{2}\right)+t_{2}\left(\frac{N_{1}-3}{2}\right).
\end{align*}
The last equality is true by Lemma \ref{t2lem1} with $n=\frac{N_{1}-1}{2}$. Our theorem is proved.
\end{proof}

\begin{lem}\label{Applem2}
We have
\begin{equation*}
S_{1}=\sum_{\begin{array}{c}
              j_{1}+2j_{2}+\ldots+nj_{n}=n \\
              j_{1}+j_{2}+\ldots+j_{n}\leq m
            \end{array}}\frac{m!}{\left(m-\sum_{\nu =1}^{n}j_{\nu}\right)!j_{1}!\cdot\ldots\cdot j_{n}!}={n+m-1\choose m-1},
\end{equation*}
where $j_{\nu}\;(1\leq\nu\leq n)$ are non-negative integers.
\end{lem}
\begin{proof}
Consider the sum
\begin{equation*}
S_{0}=\sum_{i_{1}+\ldots+i_{m}=n}1,
\end{equation*}
where $i_{\mu}\;(1\leq\mu\leq m)$ are non-negative integers. Let $\nu$ occur among $i_{\mu}$ exactly $j_{\nu}$ times, thus $S_{1}=S_{0}$. By \cite[Satz 18]{Per} we have $S_{0}= {n+m-1\choose m-1}$.
\end{proof}

\begin{thm}\label{Appthm2}
If $m>\frac{n^2}{\log 2}$, then $t_{m}(n)\neq 0$.
\end{thm}
\begin{proof}
We have
\begin{equation*}
t_{m}(n)=\sum_{i_{1}+\ldots+i_{m}=n}(-1)^{\sum_{\nu =1}^{m}s_{2}(i_{\nu})}
\end{equation*}
and by the argument used in the proof of Lemma \ref{Applem2}
\begin{equation*}
t_{m}(n)=\sum_{j_{1}+2j_{2}+\ldots+nj_{n}=n}(-1)^{\sum_{\nu =1}^{n}j_{\nu} s_{2}(\nu)}\frac{m!}{\left(m-\sum_{\nu =1}^{n}j_{\nu}\right)!j_{1}!\cdot\ldots\cdot j_{n}!}.
\end{equation*}
The summand corresponding to $j_{1}=n$ is
\begin{equation*}
(-1)^{n}\frac{m!}{(m-n)!n!}=(-1)^{n}{m\choose n}.
\end{equation*}
The sum of the absolute values of the remaining terms is by Lemma \ref{Applem2}
$$
{n+m-1\choose m-1}-{m\choose n}.
$$
Therefore, $t_{m}(n)=0$ implies
$$
{n+m-1\choose m-1}\geq 2{m\choose n}
$$
thus
$$
\frac{(m+n-1)!}{(m-1)!}\geq 2\frac{m!}{(m-n)!}
$$
and on taking logarithms
\begin{equation}\label{Appeq1}
\sum_{i=0}^{n-1}\frac{n-1}{m-n+1+i}\geq \sum_{i=0}^{n-1}\log \left(1+\frac{n-1}{m-n+1+i}\right)\geq \log 2.
\end{equation}
However,
\begin{equation}\label{Appeq2}
\sum_{i=0}^{n-1}\frac{1}{m-n+1+i}<\int_{m-n}^{m}\frac{dt}{t}=\log\frac{m}{m-n}<\frac{n}{m-n}
\end{equation}
and for $m>\frac{n^2}{\log 2}$ we obtain
$$
m-n>\frac{n^2-n}{\log 2}
$$
and from (\ref{Appeq1}) and (\ref{Appeq2}) we get
$$
\log 2> \log 2,
$$
which is impossible.
\end{proof}

\bigskip

\noindent Jagiellonian University, Faculty of Mathematics and Computer Science, Institute of Mathematics, {\L}ojasiewicza 6, 30 - 348 Krak\'{o}w, Poland;
 email: {\tt maciej.ulas@uj.edu.pl}

 \end{document}